\documentclass[12pt]{amsproc}

\usepackage{fullpage}
\usepackage[utf8]{inputenc}
\usepackage[T1]{fontenc}
\usepackage{graphicx,subcaption}
\usepackage{amsmath,amsfonts,amssymb,mathtools}
\usepackage{stmaryrd}
\usepackage{mathrsfs,dsfont}
\usepackage{amsthm}
\usepackage{mathabx}
\usepackage{tabularx}
\usepackage{float}
\usepackage{xcolor}
\usepackage{MnSymbol}
\usepackage{hyperref}
\usepackage{enumerate}

\usepackage{color}

\newcommand{\E}{\mathbb{E}}
\newcommand{\R}{\mathbb{R}}
\newcommand{\N}{\mathbb{N}}

\newcommand{\IA}{\mathcal{A}}
\newcommand{\IB}{\mathcal{B}}
\newcommand{\IL}{\Lambda}

\newcommand{\dd}{\text{d}}

\newcommand{\HH}{\mathcal{H}}
\newcommand{\EE}{\mathcal{E}}
\newcommand{\WW}{\mathcal{W}}

\newtheorem{theo}{Theorem}[section]

\newtheorem{rem}[theo]{Remark}

\newtheorem{propo}[theo]{Proposition}
\newtheorem{lemma}[theo]{Lemma}

\usepackage{algorithm}
\usepackage{algorithmic}

\newcommand{\LTexact}{\rm LT, exact}
\newcommand{\LTexpo}{\rm LT, expo}
\newcommand{\LTimp}{\rm LT, imp}

\newcommand{\AC}{\rm AC}

\begin{document}

\title{Splitting schemes for FitzHugh--Nagumo stochastic partial differential equations}

\author{Charles-Edouard Br\'ehier}
\address{Univ Lyon, Université Claude Bernard Lyon 1, CNRS UMR 5208, Institut Camille Jordan, 43 blvd. du 11 novembre 1918, F--69622 Villeurbanne cedex, France}
\email{brehier@math.univ-lyon1.fr}

\author{David Cohen}
\address{Department of Mathematical Sciences, Chalmers University of Technology and University of
Gothenburg, SE--41296 Gothenburg, Sweden}
\email{david.cohen@chalmers.se}

\author{Giuseppe Giordano}
\address{Department of Mathematics, University of Salerno, I--84084 Fisciano, Italy}
\email{gigiordano@unisa.it}

\begin{abstract}
We design and study splitting integrators for the temporal discretization of
the stochastic FitzHugh--Nagumo system. This system is a model for signal propagation in nerve cells where the voltage
variable is solution of a one-dimensional parabolic PDE with a cubic nonlinearity driven by
additive space-time white noise. We first show that the numerical solutions have finite moments.
We then prove that the splitting schemes have, at least, the strong rate of convergence $1/4$.
Finally, numerical experiments illustrating the performance of the splitting schemes are provided.
\end{abstract}

\maketitle

\section{Introduction}

The deterministic FitzHugh--Nagumo system is a simplified two-dimensional version of the famous
Hodgkin--Huxley model which describes how action potentials propagate along an axon.
Noise is omnipresent in neural systems and arises from different sources:
it could be internal noise (such as random synaptic input from other neurons) or
external noise, see for instance \cite{LINDNER2004321} for details.
It was noted in \cite{sp05noise} that the addition of an appropriate amount
of noise in the model helps to detect weak signals.
All this has attracted a large body of works on
the analysis of the influence of external random perturbations in neurons in the recent years, see for instance \cite{MR4282767,LINDNER2004321,MR2661971,deepak,MR4249991,sp05noise,MR3051033,MR3051032,MR3982703}.

In this article, we consider the stochastic FitzHugh--Nagumo system
\begin{equation*}
\left\lbrace
\begin{aligned}
&\frac\partial{\partial t} u(t,\zeta)=\frac{\partial^2}{\partial \zeta^2} u(t,\zeta)+u(t,\zeta)-u^3(t,\zeta)-v(t,\zeta)+ \frac{\partial^2}{\partial t\partial\zeta} W(t,\zeta),\\
&\frac\partial{\partial t} v(t,\zeta)=\gamma_1 u(t,\zeta)-\gamma_2 v(t,\zeta)+\beta,\\
&\frac\partial{\partial \zeta}u(t,0)=\frac\partial{\partial \zeta}u(t,1)=0,\\
&u(0,\zeta)=u_0(\zeta),v(0,\zeta)=v_0(\zeta),
\end{aligned}
\right.
\end{equation*}
for $\zeta\in(0,1)$ and $t\geq0$.
The objective of this article is to design and analyse numerical integrators, which treat explicitly the nonlinearity, for the temporal discretization of the system above, based on splitting strategies.

In the stochastic partial differential equation (SPDE) above, the unknowns $u=\bigl(u(t)\bigr)_{t\ge 0}$ and $v=\bigl(v(t)\bigr)_{t\ge 0}$ are $L^2(0,1)$-valued stochastic processes, with initial values $u_0,v_0\in L^2(0,1)$,
see~Section~\ref{sec:setting} and the standard monograph~\cite{MR3236753}
on stochastic evolution equations in Hilbert spaces.
In addition, $\gamma_1,\gamma_2,\beta\in\mathbb R$ are three real-valued parameters,
$\Delta=\frac{\partial^2}{\partial\zeta^2}$ is the Laplace operator endowed with homogeneous Neumann boundary conditions, and $\bigl(W(t)\bigr)_{t\geq0}$ is a cylindrical Wiener process, meaning that the component $u$ is driven by space-time white noise. The component $u$ represents the voltage variable while the component $v$ the recovery variable. The noise represents random fluctuations of the membrane potential, see~\cite{sp05noise} for a related model with a scalar noise. Note that in the considered system only the evolution of the voltage variable $u$ is driven by a Wiener process. Having noise for the evolution of the recovery variable $v$ would correspond to modelling different biological phenomena which are not treated in this work.

The major difficulty in the theoretical and numerical analysis of the SPDE system above is the nonlinearity $u-u^3$ appearing in the evolution of the component $u$: this nonlinearity is not globally Lipschitz continuous and has polynomial growth. As proved in~\cite{arnulf}, using a standard explicit discretization like the Euler--Maruyama method would yield numerical schemes which usually do not converge: more precisely, moment bounds, uniform with respect to the time step size, would not hold for such methods.

For an efficient numerical simulation of the above SPDE system, we propose to exploit a splitting strategy to define integrators and we show that appropriate moment bounds and strong error estimates can be obtained.
In a nutshell, the main idea of a splitting strategy is to decompose the vector field,
appearing in the evolution equation, in several parts, in order to exhibit subsystems which can be integrated exactly (or easily). One then composes the (exact or approximate) flows associated with the subsystems to define integrators applied to the original problem.  Splitting schemes have a long history in the numerical analysis of ordinary and partial differential equations, see for instance \cite{MR3642447,MR2840298,MR2132573,MR2009376} and references therein. Splitting integrators have recently been applied and analysed in the context of stochastic ordinary and partial differential equations.
Without being exhaustive, we refer the interested reader to
\cite{MR3511359,MR3570281,MR2608370,bct21,cv21,MR3362507,MR1860968} for the finite-dimensional context and to   \cite{MR4263224,bc20,MR4400428,MR4019051,MR4132896,MR2646103,MR3912762,MR3607207,MR3021492,MR3119724,m06,MR4278943,MR3839068} for the context of SPDEs.

The main result of this paper is a strong convergence result, with rate of convergence $1/4$,
for easy to implement splitting integrators, see Equation~\eqref{eq:LTscheme} in Subsection~\ref{sec-splitting},
for the time discretization of the SPDE defined above, see Theorem~\ref{theo:error}
for a precise statement.
To the best of our knowledge, Theorem~\ref{theo:error} is the first strong convergence result obtained for
a time discretization scheme applied to the stochastic FitzHugh--Nagumo SPDE system.
The first non-trivial step of the analysis is to obtain suitable moment bounds for the splitting scheme, see Theorem~\ref{theo:momentbounds}. Note that the proof of the moment bounds of Theorem~\ref{theo:momentbounds} is inspired by the article~\cite{MR3986273} where splitting schemes for the
stochastic Allen--Cahn equation
\[
\text{d} u(t)=\Delta u(t)\,\text{d} t+(u(t)-u^3(t))\,\text{d} t+ \text{d} W(t) 
\]
were studied. The proof of the strong convergence error estimates of Theorem~\ref{theo:error} is inspired by the article~\cite{MR4019051}. However, one needs a dedicated and detailed analysis since the considered stochastic FitzHugh--Nagumo system is not a parabolic stochastic evolution system, and several arguments are non trivial. Note also that the construction of the splitting scheme is inspired by the recent article \cite{BUCKWAR2022} which treats a finite dimensional version
\begin{equation*}
\left\lbrace
\begin{aligned}
&\text{d} u(t)=(u(t)-u^3(t)-v(t))\,\text{d} t,\\
&\text{d} v(t)=(\gamma_1 u(t)-\gamma_2 v(t)+\beta)\,\text{d} t+\text{d} B(t),\\
&u(0)=u_0,v(0)=v_0,
\end{aligned}
\right.
\end{equation*}
of the stochastic FitzHugh--Nagumo system (where the finite-dimensional noise $B$ is in the $v$-component).

We now review the literature related to this work. The recent article \cite{BUCKWAR2022} analyses the strong convergence of
splitting schemes for a class of semi-linear stochastic differential equations (SDEs)
as well as preservation of possible structural properties of the problem. Applications to the proposed schemes
to the stochastic FitzHugh--Nagumo SDE are also presented.
The work \cite{MR2467646} performs extensive numerical simulations on the FitzHugh--Nagumo equation with space-time white noise in $1d$.
A finite difference discretization is used in space, while the classical Euler--Maruyama is used in time.
The article \cite{MR3394442} studies numerically the FitzHugh--Nagumo equation with colored noise in $2d$.
In particular, the authors use a finite element discretization in space and the semi-implicit Euler--Maruyama scheme in time. The two previously mentioned works employ crude explicit discretization for the nonlinearity and
therefore may have the issues about moment bounds discussed above.
The work \cite{MR3498982} proves convergence (without rates) of a fully-discrete numerical scheme, based on a Galerkin method in space and the tamed Euler scheme in time,
for a general SPDE with super-linearly growing operators. This is then applied to the FitzHugh--Nagumo equation with space-time white noise in $1d$.
The articles~\cite{MR3290962} and~\cite{MR3511288} prove strong convergence rates of a finite difference spatial
discretization of the FitzHugh--Nagumo equation
with space-time white noise in $1d$.

This article is organized as follows. The setting is given in Section~\ref{sec:setting}, in particular this allows us to state a well-posedness result for the considered stochastic FitzHugh--Nagumo system. The splitting strategy, the proposed integrators and the main results of the paper are then presented in Sections~\ref{sec-subsystem},~\ref{sec-splitting} and~\ref{sec-main} respectively.
Several auxiliary results are stated and proved
in Section~\ref{sec-prel}. Section~\ref{sectionProofs} gives the proofs of Theorems~\ref{theo:momentbounds} and~\ref{theo:error}. Finally, numerical experiments are provided in Section~\ref{sect-num}.

\section{Setting}\label{sec:setting}

This section is devoted to introducing the functional framework, the linear and nonlinear operators, and the Wiener process. This allows us to consider the stochastic FitzHugh--Nagumo SPDE system as a stochastic evolution equation in the classical framework of \cite{MR3236753}.

\subsection{Functional framework}

Let us first introduce the infinite-dimensional, separable Hilbert space $H=L^2(0,1)$ of square integrable functions from $(0,1)$ to $\R$.
This space is equipped with the inner product $\langle\cdot,\cdot\rangle_H$ and the norm $\|\cdot\|_H$ which satisfy
\[
\langle u_1,u_2\rangle_H=\displaystyle\int_0^1 u_1(\zeta)u_2(\zeta)\,\dd \zeta,\quad\|u\|_{H}=\sqrt{\langle u,u\rangle_H},
\]
respectively, for all $u_1,u_2,u\in H$. Let us then introduce the product space $\HH=H\times H$, which is also an infinite-dimensional, separable Hilbert space, with the inner product $\langle\cdot,\cdot\rangle_\HH$ and the norm $\|\cdot\|_\HH$ defined by
\[
\langle x_1,x_2\rangle_\HH=\langle u_1,u_2\rangle_H+\langle v_1,v_2\rangle_H,\quad\|x\|_\HH=\sqrt{\|u\|_H^2+\|v\|_H^2},
\]
for all $x_1=(u_1,v_1),x_2=(u_2,v_2),x=(u,v)\in\HH$.

Let also $E=\mathcal{C}^0([0,1])$ be the space of continuous functions from $[0,1]$ to $\R$, and set $\mathcal{E}=E\times E$. Then $E$ and $\EE$ are separable Banach spaces, with the norms $\|\cdot\|_E$ and $\|\cdot\|_\EE$ defined by
\[
\|u\|_E=\max_{\zeta\in[0,1]}|u(\zeta)|,\quad \|x\|_\EE=\max\bigl(\|u\|_E,\|v\|_E\bigr)
\]
for all $u\in E$ and $x=(u,v)\in\EE$.

Let us denote the inner product and the norm in the finite-dimensional Euclidean space $\R^2$ by $\langle\cdot,\cdot\rangle$ and $\|\cdot\|$ respectively. If $M$ is a $2\times2$ real-valued matrix, let $\vvvert M\vvvert=\underset{x\in\R^2;~\|x\|=1}\sup~\|Mx\|$.

Finally, in the sequel, $\N=\{1,2,\ldots\}$ denotes the set of integers and $\N_0=\{0\}\cup\N=\{0,1,\ldots\}$ denotes the set of nonnegative integers. We often write $j\ge 1$ (resp. $j\ge 0$) instead of $j\in\N$ (resp. $j\in\N_0$).

\subsection{Linear operators}

This subsection presents the material required to use the semigroup approach for SPDEs, see for instance \cite{MR3236753}.

For all $j\in\N$, set $\lambda_j=(j\pi)^2$ and $e_j(\zeta)=\sqrt{2}\cos(j\pi\zeta)$ for all $\zeta\in[0,1]$. In addition, set $\lambda_0=0$ and $e_0(\zeta)=1$ for all $\zeta\in[0,1]$. Then $\bigl(e_j\bigr)_{j\ge 0}$ is a complete orthonormal system of $H$, and one has
\[
\Delta e_j=-\lambda_je_j
\]
for all $j\ge 0$, where $\Delta$ denotes the Laplace operator with homogeneous Neumann boundary conditions. For all $u\in H$ and all $t\ge 0$, set
\begin{equation}\label{eq:expoDelta}
e^{t\Delta}u=\sum_{j\ge 0}e^{-t\lambda_j}\langle u,e_j\rangle_{H} e_j.
\end{equation}
Then, for any $u_0\in H$, the mapping $(t,\zeta)\mapsto u(t,\zeta)=e^{t\Delta}u_0(\zeta)$ is the unique solution of the heat equation on $(0,1)$ with homogeneous Neumann boundary conditions and initial value $u(0,\cdot)=u_0$:
\[
\left\lbrace
\begin{aligned}
&\frac{\partial u(t,\zeta)}{\partial t}=\Delta u(t,\zeta),\quad t>0,~\zeta\in(0,1),\\
&\frac{\partial u(t,0)}{\partial\zeta} =\frac{\partial u(t,1)}{\partial\zeta} =0,\quad t>0,\\
&u(0,\zeta)=u_0(\zeta),\quad \zeta\in(0,1).
\end{aligned}
\right.
\]
For all $\alpha\in[0,2]$, set
\begin{align*}
&H^\alpha=\left\{u\in H;~\sum_{j\ge 0}\lambda_j^{\alpha}\langle u,e_j\rangle_{H}^2<\infty\right\},\\
&(-\Delta)^{\frac{\alpha}{2}} u=\sum_{j\ge 0}\lambda_j^{\frac{\alpha}{2}}\langle u,e_j\rangle_{H} e_j,\quad u\in H^\alpha.
\end{align*}
Observe that $H^0=H=L^2(0,1)$.
The Laplace operator $\Delta$ with homogeneous Neumann boundary conditions is a self-adjoint unbounded linear operator on $H$, with domain $D(\Delta)=H^2$. We also let
$\HH^\alpha=H^\alpha\times H$ for all $\alpha\in[0,2]$.

Let us now introduce the linear operator $\IL$, defined as follows: for all $x=(u,v)\in \HH^2$, set
\[
\IL x=\begin{pmatrix} -\Delta u\\ 0\end{pmatrix}.
\]
Then $\IL$ is a self-adjoint unbounded linear operator on $\HH$, with domain $D(\IL)=\HH^2$. For all $x=(u,v)\in\HH$ and $t\ge 0$, set
\begin{equation}\label{eq:expoIL}
e^{-t\IL}x=\begin{pmatrix} e^{t\Delta}u\\ v\end{pmatrix}.
\end{equation}
Regularity estimates for this operator are presented in Section~\ref{sec-prel} below.

\subsection{Nonlinear operator}
Let $\beta,\gamma_1,\gamma_2\in\R$ be parameters of the model. Define the mapping $F:\R^2\to\R^2$ such that for all $x=(u,v)\in\R^2$ one has
\[
F(x)=\begin{pmatrix} u-u^3-v\\ \gamma_1 u- \gamma_2v+\beta\end{pmatrix}.
\]
In order to define splitting schemes, it is convenient to introduce two auxiliary mappings $F^{\rm NL}:\R^2\to\R^2$ and $F^{\rm L}:\R^2\to\R^2$ defined as follows: for all $x=(u,v)\in\R^2$, set
\begin{align*}
&F^{\rm NL}(x)=\begin{pmatrix} u-u^3 \\ \beta\end{pmatrix} \\
&F^{\rm L}(x)=\begin{pmatrix} -v\\ \gamma_1 u-\gamma_2v\end{pmatrix}=Bx,
\end{align*}
where the matrix $B$ is defined by
\[
B=\begin{pmatrix} 0 & -1\\ \gamma_1 & -\gamma_2\end{pmatrix}.
\]
One then has
\begin{equation}\label{eq:F}
F(x)=F^{\rm NL}(x)+F^{\rm L}(x)
\end{equation}
for all $x\in\R^2$. The mapping $F^{\rm L}$ is globally Lipschitz continuous: for all $x_1,x_2\in\R^2$ one has
\[
\|F^{\rm L}(x_2)-F^{\rm L}(x_1)\|\le \vvvert B\vvvert \|x_2-x_1\|.
\]
However $F$ and $F^{\rm NL}$ are only locally Lipschitz continuous, and satisfy a one-sided Lipschitz continuity property: there exists $C\in(0,\infty)$ such that for all $x_1,x_2\in\R^2$ one has
\begin{equation}\label{eq:onesidedF}
\langle x_2-x_1,F^{\rm NL}(x_2)-F^{\rm NL}(x_1)\rangle\le C\|x_2-x_1\|^2,\quad \langle x_2-x_1,F(x_2)-F(x_1)\rangle\le C\|x_2-x_1\|^2.
\end{equation}
In the sequel, an abuse of notation is used for simplicity: the same notation is employed for a mapping $f:\R^2\to\R^2$ and for the associated Nemytskii operator defined on $\HH$ or on $\EE$ by $f(u,v)=f(u(\cdot),v(\cdot))$.

\subsection{Wiener process}

It remains to define the noise that drives the stochastic FitzHugh--Nagumo system. Let $\bigl(W(t)\bigr)_{t\ge 0}$ be a cylindrical Wiener process on $H$: given a sequence $\bigl(\beta_j(\cdot)\bigr)_{j\ge 0}$ of independent standard real-valued Wiener processes, defined on a probability space $(\Omega,\mathcal{F},\mathbb{P})$ equipped with a filtration $\bigl(\mathcal{F}_t\bigr)_{t\ge 0}$ which satisfies the usual conditions
and where $\E[\cdot]$ denotes the expectation operator on the probability space, set
\begin{equation}
W(t)=\sum_{j\ge 0}\beta_j(t)e_j.
\end{equation}
For all $t\ge 0$, define
\[
\WW(t)=\begin{pmatrix} W(t)\\ 0\end{pmatrix}=\sum_{j\ge 0}\beta_j(t)\begin{pmatrix} e_j\\0\end{pmatrix},
\]
then $\bigl(\WW(t)\bigr)_{t\ge 0}$ is a generalized $\mathcal{Q}$-Wiener process on $\HH$, with the covariance operator
\[
\mathcal{Q}=\begin{pmatrix}I & 0\\ 0 & 0\end{pmatrix}.
\]
Note that almost surely $W(t)\notin H$ and $\WW(t)\notin \HH$ for all $t>0$. However, for all $T\ge 0$, the It\^o stochastic integrals $\int_0^T L(t)\,\dd W(t)$ and $\int_0^T \mathcal{L}(t)\,\dd \WW(t)$ are well-defined $H$-valued and $\HH$-valued random variables respectively, if $\bigl(L(t)\bigr)_{0\le t\le T}$ and $\bigl(\mathcal{L}(t)\bigr)_{0\le t\le T}$ are adapted processes which satisfy
$\sum_{j\ge 0}\int_0^T \E[\|L(t)e_j\|_H^2]\,\dd t<\infty$ and $\sum_{j\ge 0}\int_0^T \E[\|\mathcal{L}(t)\begin{pmatrix}e_{j}\\ 0\end{pmatrix}\|_{\HH}^2]\,\dd t<\infty$ respectively.

Observe that for all $T\ge 0$ one has
\[
\sum_{j\ge 0} \int_{0}^{T}\|e^{t\Delta}e_j\|_H^2\,\dd t=\sum_{j\ge 0} \int_{0}^{T}\|e^{-t\IL}\begin{pmatrix}e_j\\0\end{pmatrix}\|_\HH^2\,\dd t\le T+\sum_{j\ge 1}\lambda_j^{-1}<\infty.
\]
Therefore, for all $t\ge 0$ one can define the $H$-valued random variable $Z(t)$ and the $\HH$-valued random variable $\mathcal{Z}(t)$, called
the stochastic convolutions, by
\begin{equation}\label{eq:Z}
\begin{aligned}
Z(t)&=\int_0^t e^{(t-s)\Delta}\,\dd W(s),\\
\mathcal{Z}(t)&=\int_0^t e^{-(t-s)\IL}\,\dd \WW(s).
\end{aligned}
\end{equation}
The processes $\bigl(Z(t)\bigr)_{t\ge 0}$ and $\bigl(\mathcal{Z}(t)\bigr)_{t\ge 0}$ are interpreted as the mild solutions of the stochastic evolution equations
\begin{align*}
\dd Z(t)&=\Delta Z(t)\,\dd t+\dd W(t),\\
\dd \mathcal{Z}(t)&=-\IL\mathcal{Z}(t)\,\dd t+\dd \WW(t)
\end{align*}
with initial values $Z(0)=0$ and $\mathcal{Z}(0)=0$. Note that $\mathcal{Z}(t)=\begin{pmatrix}Z(t)\\0\end{pmatrix}$ for all $t\ge 0$.

\subsection{The stochastic FitzHugh--Nagumo SPDE system}

In this work, we study numerical schemes for the FitzHugh--Nagumo stochastic system for signal propagation in nerve cells.
This system is written as the stochastic evolution system
\begin{equation}
\label{eq:FhNsystem}
\left\lbrace
\begin{aligned}
&\dd u(t)=\Delta u(t)\,\dd t+(u(t)-u^3(t)-v(t))\,\dd t+ \dd W(t),\\
&\dd v(t)=(\gamma_1 u(t)-\gamma_2v(t)+\beta)\,\dd t,\\
&u(0)=u_0,v(0)=v_0,
\end{aligned}
\right.
\end{equation}
where the unknowns $u(\cdot)=\bigl(u(t)\bigr)_{t\ge 0}$ and $v(\cdot)=\bigl(v(t)\bigr)_{t\ge 0}$ are $H$-valued stochastic processes, and with initial values $u_0\in H$ and $v_0\in H$.
Recall that Neumann boundary conditions are used in the above system.
Using the notation introduced above and setting $X(t)=(u(t),v(t))$ for all $t\ge 0$, the stochastic evolution system~\eqref{eq:FhNsystem} is treated in the sequel as the stochastic evolution equation
\begin{equation}
\label{eq:SEE}
\dd X(t)=-\IL X(t)\,\dd t+F(X(t))\,\dd t+\dd \WW(t),\, X(0)=x_0,
\end{equation}
with the initial value $x_0=(u_0,v_0)\in\HH$. For all $T\in(0,\infty)$, a stochastic process $\bigl(X(t)\bigr)_{0\le t\le T}$ is called a mild solution of~\eqref{eq:SEE} if it has continuous trajectories with values in $\HH$, and if for all $t\in[0,T]$ one has
\begin{equation}\label{eq:SEEmild}
X(t)=e^{-t\IL}x_0+\int_0^t e^{-(t-s)\IL}F(X(s))\,\dd s+\int_0^t e^{-(t-s)\IL}\,\dd \WW(s).
\end{equation}
In the framework presented in this section, the stochastic evolution equation~\eqref{eq:SEE} admits a unique global mild solution,
for any initial value $x_0\in\HH^{2\alpha}\cap \EE$ and for $\alpha\in[0,\frac14)$, see Proposition~\ref{propo:SEE} below.

For simplicity, the initial values $u_0,v_0$, resp. $x_0$, appearing in~\eqref{eq:FhNsystem}, resp.~\eqref{eq:SEE}, are deterministic. It would be straightforward to extend the results below for random initial values which are independent of the Wiener process and are assumed to satisfy appropriate moment bounds, using a conditioning argument.

\section{Splitting schemes}\label{sec:splitt}

The time-step size of the integrators defined below is denoted by $\tau$. Without loss of generality, it is assumed that $\tau\in(0,\tau_0)$,
where $\tau_0$ is an arbitrary positive real number, and that there exists $T\in(0,\infty)$ and $N\in\N$ such that $\tau=T/N$. The notation $t_n=n\tau$ for $n\in\{0,\ldots,N\}$ is used in the sequel. The increments of the Wiener processes $\bigl(W(t)\bigr)_{t\ge 0}$ and $\bigl(\WW(t)\bigr)_{t\ge 0}$ are denoted by
\[
\delta W_n=W(t_{n+1})-W(t_n),\quad \delta \WW_n=\WW(t_{n+1})-\WW(t_n)=\begin{pmatrix} \delta W_n\\0 \end{pmatrix}.
\]
The proposed time integrators for the SPDE~\eqref{eq:SEE} are based on a splitting strategy.
Recall that the main principle of splitting integrators is to decompose the vector field
of the evolution problem in several parts, such that the arising subsystems are exactly (or easily) integrated.
We define these subsystems in Subsection~\ref{sec-subsystem}, then give the definitions of
three splitting schemes in Subsection~\ref{sec-splitting} and state the main results of this article in Subsection~\ref{sec-main}.

\subsection{Solutions of auxiliary subsystems}\label{sec-subsystem}

The construction of the proposed splitting schemes is based
on the combination of exact or approximate solutions of the three subsystems considered below.

$\bullet$ The nonlinear differential equation (considered on the Euclidean space $\R^2$)
\begin{equation}\label{eq:subODE-NL}
\left\lbrace
\begin{aligned}
&\frac{\dd x^{\rm NL}(t)}{\dd t}=F^{\rm NL}(x^{\rm NL}(t)),\\
&x^{\rm NL}(0)=x_0\in\R^2
\end{aligned}
\right.
\end{equation}
admits a unique global solution $\bigl(x^{\rm NL}(t)\bigr)_{t\ge 0}$.
This solution has the following exact expression, see for instance~\cite[Equation~(3)]{MR3986273}: for all $t\ge 0$ and $x_0=(u_0,v_0)\in\R^2$, one has
\begin{equation}\label{eq:phi-NL}
x^{\rm NL}(t)=\phi_t^{\rm NL}(x_0)=\begin{pmatrix} \frac{u_0}{\sqrt{u_0^2+(1-u_0^2)e^{-2t}}} \\ v_0+\beta t\end{pmatrix}.
\end{equation}

$\bullet$ The linear differential equation (considered on the Euclidean space $\R^2$)
\begin{equation}\label{eq:subODE-L}
\left\lbrace
\begin{aligned}
&\frac{\dd x^{\rm L}(t)}{\dd t}=F^{\rm L}(x^{\rm L}(t)),\\
&x^{\rm L}(0)=x_0\in\R^2
\end{aligned}
\right.
\end{equation}
admits a unique global solution $\bigl(x^{\rm L}(t)\bigr)_{t\ge 0}$.
This solution has the following expression: for all $t\ge 0$ and $x_0=(u_0,v_0)\in\R^2$, one has
\begin{equation}\label{eq:phi-L}
x^{\rm L}(t)=\phi_t^{\rm L}(x_0)=e^{tB}x_0.
\end{equation}

$\bullet$ The stochastic evolution equation (considered on the Hilbert space $\HH$)
\begin{equation}\label{eq:subSPDE}
\left\lbrace
\begin{aligned}
&\dd X^{\rm s}(t)=-\IL X^{\rm s}(t)\,\dd t+\dd \WW(t)\\
&X^{\rm s}(0)=x_0\in\HH
\end{aligned}
\right.
\end{equation}
admits a unique global solution $\bigl(X^{\rm s}(t)\bigr)_{t\ge 0}$. This solution
has the following expression: for all $t\ge 0$ and $x_0=(u_0,v_0)\in\HH$, one has
\begin{equation}\label{eq:sol-subSPDE}
X^{\rm s}(t)=e^{-t\IL}x_0+\int_0^t e^{-(t-s)\IL}\dd\WW(s)=\begin{pmatrix} e^{t\Delta}u_0+\int_0^t e^{(t-s)\Delta}\,\dd W(s)\\ v_0\end{pmatrix},
\end{equation}
see~\eqref{eq:Z} for the expression of the stochastic convolution.
For all $n\in\{0,\ldots,N-1\}$,
set $X^{\rm s,exact}_{n}=X^{\rm s}(t_{n})$, then one has the following recursion formula
\begin{equation}\label{eq:sol-subSPDE-local}
X^{\rm s,exact}_{n+1}=e^{-\tau\IL}X^{\rm s,exact}_n+\int_{t_n}^{t_{n+1}}e^{-(t_{n+1}-s)\IL}\,\dd\WW(s)
\end{equation}
recalling the notation $t_n=n\tau$.

Instead of using the exact solution~\eqref{eq:sol-subSPDE} of the stochastic convolution~\eqref{eq:subSPDE}, one can use approximate
solutions $\bigl(X_n^{\rm s,exp}\bigr)_{n\ge 0}=\bigl(u_n^{\rm s,exp},v_n^{\rm s,exp}\bigr)_{n\ge 0}$ and $\bigl(X_n^{\rm s,imp}\bigr)_{n\ge 0}=\bigl(u_n^{\rm s,imp},v_n^{\rm s,imp}\bigr)_{n\ge 0}$ defined by an exponential Euler scheme
and a linear implicit Euler scheme respectively:
\begin{equation}\label{eq:exp-subSPDE}
X_{n+1}^{\rm s,exp}=e^{-\tau\IL}\Bigl(X_n^{\rm s,exp}+\delta\WW_n\Bigr)=\begin{pmatrix} e^{\tau\Delta}\bigl(u_n^{\rm s,exp}+\delta W_n\bigr)\\ v_n^{\rm s,exp}\end{pmatrix},
\end{equation}
and
\begin{equation}\label{eq:imp-subSPDE}
X_{n+1}^{\rm s,imp}=\bigl(I+\tau\IL\bigr)^{-1}\Bigl(X_n^{\rm s,imp}+\delta\WW_n\Bigr)=\begin{pmatrix} (I-\tau\Delta)^{-1}\bigl(u_n^{\rm s,imp}+\delta W_n\bigr)\\ v_n^{\rm s,imp}\end{pmatrix},
\end{equation}
with initial values $X_0^{\rm s,exp}=X_0^{\rm s,imp}=x_0=(u_0,v_0)\in\HH$, $u_0^{\rm s,exp}=u_0^{\rm s,imp}=u_0\in H$ and $v_0^{\rm s,exp}=v_0^{\rm s,imp}=v_0\in H$.

\subsection{Definition of the splitting schemes}\label{sec-splitting}

We are now in position to introduce the three splitting schemes studied in this article.
They are constructed using a Lie--Trotter strategy, where first the subsystems~\eqref{eq:subODE-NL},~\eqref{eq:subODE-L} are solved exactly using the flow maps~\eqref{eq:phi-NL} and~\eqref{eq:phi-L} respectively, and where the subsystem~\eqref{eq:subSPDE} is either solved exactly using~\eqref{eq:sol-subSPDE} or approximately using~\eqref{eq:exp-subSPDE} or~\eqref{eq:imp-subSPDE}.

For the composition of the first two subsystems, define the mapping $\phi_\tau:\R^2\to\R^2$ as follows: for all $\tau\in(0,\tau_0)$, set
\begin{equation}\label{eq:phitau}
\phi_\tau=\phi_\tau^{\rm L}\circ\phi_\tau^{\rm NL}.
\end{equation}
Using the expression~\eqref{eq:sol-subSPDE-local} for the exact solution~\eqref{eq:sol-subSPDE} of~\eqref{eq:subSPDE} leads to the definition of
the following explicit splitting scheme for the stochastic FitzHugh--Nagumo SPDE system~\eqref{eq:FhNsystem}:
\begin{equation}\label{eq:LTexact}
X_{n+1}^{\LTexact}=e^{-\tau\IL}\phi_\tau\bigl(X_n^{\LTexact}\bigr)+\int_{t_n}^{t_{n+1}}e^{-(t_{n+1}-s)\IL}\,\dd\WW(s).
\end{equation}
Using the exponential Euler scheme~\eqref{eq:exp-subSPDE} to approximate the solution of~\eqref{eq:subSPDE} leads to the definition of the following explicit splitting
scheme for~\eqref{eq:FhNsystem}:
\begin{equation}\label{eq:LTexpo}
X_{n+1}^{\LTexpo}=e^{-\tau\IL}\phi_\tau\bigl(X_n^{\LTexpo}\bigr)+e^{-\tau\IL}\delta\WW_n.
\end{equation}

Using the linear implicit Euler scheme~\eqref{eq:imp-subSPDE} to approximate the solution of~\eqref{eq:subSPDE} leads to the definition of the following splitting
scheme for~\eqref{eq:FhNsystem}:
\begin{equation}\label{eq:LTimp}
X_{n+1}^{\LTimp}=(I+\tau\IL)^{-1}\phi_\tau\bigl(X_n^{\LTimp}\bigr)+(I+\tau\IL)^{-1}\delta\WW_n.
\end{equation}
For these three Lie--Trotter splitting schemes~\eqref{eq:LTexact},~\eqref{eq:LTexpo} and~\eqref{eq:LTimp}, the same initial value is imposed:
\[
X_0^{\LTexact}=X_0^{\LTexpo}=X_0^{\LTimp}=x_0\in\HH.
\]

Before proceeding with the statements of the main results, let us give several observations and auxiliary tools.

Observe that the three schemes~\eqref{eq:LTexact},~\eqref{eq:LTexpo} and~\eqref{eq:LTimp} can be written using the single formulation
\begin{equation}\label{eq:LTscheme}
X_{n+1}=\IA_\tau\phi_\tau(X_n)+\int_{t_n}^{t_{n+1}}\IB_{t_{n+1}-s}\,\dd\WW(s)
\end{equation}
which is used in the analysis below.
The expressions of the linear operators $\IA_\tau$ and $\IB_{t_{n+1}-s}$ for each of the three schemes are given by:
$\IA_\tau=e^{-\tau\IL}, \IB_{t_{n+1}-s}=e^{-(t_{n+1}-s)\IL}$ for the scheme~\eqref{eq:LTexact}
$\IA_\tau=\IB_{t_{n+1}-s}=e^{-\tau\IL}$ for the scheme~\eqref{eq:LTexpo},
and $\IA_\tau=\IB_{t_{n+1}-s}=(I+\tau\IL)^{-1}$ for the scheme~\eqref{eq:LTimp}.

For any value $\tau\in(0,\tau_0)$ of the time-step size, introduce the mapping $\psi_{\tau}:\R^2\to\R^2$ defined as follows: for all $x\in\R^2$,
\begin{equation}\label{eq:psitau}
\psi_\tau(x)=\frac{\phi_\tau(x)-x}{\tau}.
\end{equation}
The Lie--Trotter splitting scheme~\eqref{eq:LTscheme} is then written as
\[
X_{n+1}=\IA_\tau X_n+\tau\IA_\tau\psi_\tau(X_n)+\int_{t_n}^{t_{n+1}}\IB_{t_{n+1}-s}\,\dd\WW(s)
\]
and can thus be interpreted as a numerical scheme applied to the auxiliary stochastic evolution equation
\begin{equation}\label{eq:SEEaux}
\dd X_\tau(t)=-\IL X_\tau(t)\,\dd t+\psi_\tau(X_\tau(t))\,\dd t+\dd \WW(t),\, X_\tau(0)=x_0.
\end{equation}
Note that the SPDE~\eqref{eq:SEEaux} is similar to the original problem~\eqref{eq:SEE},
however the nonlinearity $F$ is replaced by the auxiliary mapping $\psi_\tau$.

\subsection{Main results}\label{sec-main}
In this subsection, we state the main results of this article. First, we give moment bounds for
the three splitting schemes~\eqref{eq:LTscheme}, see Theorem~\ref{theo:momentbounds}. Then, we give
strong error estimates, with rate of convergence $1/4$, for the numerical approximations
of the solution of the stochastic FitzHugh--Nagumo SPDE system~\eqref{eq:SEE}, see Theorem~\ref{theo:error}.

\begin{theo}\label{theo:momentbounds}
For all $T\in(0,\infty)$ and $p\in[1,\infty)$, there exists $C_p(T)\in(0,\infty)$ such that for all $x_0\in\EE$ one has
\begin{equation}\label{eq:momentboundsLTscheme}
\underset{\tau\in(0,\tau_0)}\sup~\underset{0\le n\le N}\sup~\E[\|X_n\|_\EE^p]\le C_p(T)\bigl(1+\|x_0\|_\EE^p\bigr),
\end{equation}
where $\bigl(X_n\bigr)_{n\ge 0}$ is given by~\eqref{eq:LTscheme} (with initial value $X_0=x_0$), and where $T=N\tau$ with $N\in\N$.
\end{theo}
The proof of this theorem is postponed to Section~\ref{sectionProofs}.
\begin{rem}
The nonlinear mapping $F$ is not globally Lipschitz continuous and has polynomial growth. Therefore,
if one employs a standard implicit-explicit
scheme applied directly to the original SPDE
\[
\mathcal{X}_{n+1}=\IA_\tau \mathcal{X}_n+\tau\IA_\tau F(\mathcal{X}_n)+\int_{t_n}^{t_{n+1}}\IB_{t_{n+1}-s}\,\dd\WW(s)
\]
with $\mathcal{X}_0=x_0$, where the same notation as for the scheme~\eqref{eq:LTscheme} is used, one has
\[
\underset{\tau\in(0,\tau_0)}\sup~\underset{0\le n\le N}\sup~\E[\|\mathcal{X}_n\|_\EE^p]=\underset{\tau\in(0,\tau_0)}\sup~\underset{0\le n\le N}\sup~\E[\|\mathcal{X}_n\|_\HH^p]=\infty,
\]
see for instance~\cite{arnulf} for the stochastic Allen--Cahn equation and~\cite{MR3498982}. 
As a consequence Theorem~\ref{theo:momentbounds} is not a trivial result and illustrates the superiority of the proposed explicit splitting scheme
compared with a crude explicit discretization method.
\end{rem}

We are now in position to state our strong convergence result. Its proof is given in Section~\ref{sectionProofs}.
\begin{theo}\label{theo:error}
For all $T\in(0,\infty)$, $p\in[1,\infty)$ and $\alpha\in[0,\frac14)$, there exists $C_{\alpha,p}(T)\in(0,\infty)$ such that for all $x_0=(u_0,v_0)\in\HH^{2\alpha}\cap\EE$, all $\tau\in(0,\tau_0)$, one has
\begin{equation}\label{eq:error}
\underset{0\le n\le N}\sup~\bigl(\E[\|X(t_n)-X_n\|_\HH^p]\bigr)^{\frac1p}\le C_{\alpha,p}(T)\tau^{\alpha}\bigl(1+\|(-\Delta)^\alpha u_0\|_H^7+\|x_0\|_\EE^7\bigr).
\end{equation}
\end{theo}
The order of convergence $1/4$ obtained in Theorem~\ref{theo:error} is consistent with the temporal H\"older regularity property of the trajectories $t\mapsto X(t)\in\HH$. It is also consistent with the strong convergence rate obtained in~\cite{MR4019051} for the stochastic Allen--Cahn equation. However new arguments are required
to study the FitzHugh--Nagumo system which is not a parabolic SPDE problem, and which has a cubic nonlinearity.

Let us state two of the main auxiliary results which are used in the proofs of the main results.
These propositions are proved in Subsection~\ref{cestlaquonprouve}.
\begin{propo}\label{propo:phitau}
For all $\tau\in(0,\tau_0)$, the mapping $\phi_\tau:\R^2\to\R^2$ defined by~\eqref{eq:phitau} is globally Lipschitz continuous.
In addition, for all $\tau\in(0,\tau_0)$ and all $x_1,x_2\in\R^2$ one has
\begin{equation}\label{eq:phitau-Lip}
\|\phi_\tau(x_2)-\phi_\tau(x_1)\|\le e^{(1+\vvvert B\vvvert)\tau}\|x_2-x_1\|.
\end{equation}
\end{propo}

\begin{propo}\label{propo:psitau}
There exists $C(\tau_0)\in(0,\infty)$ such that for all $\tau\in(0,\tau_0)$, the mapping $\psi_\tau:\R^2\to\R^2$ defined by~\eqref{eq:psitau} satisfies the following properties: for all $x_1,x_2\in\R^2$, one has
\begin{align}
&\langle x_2-x_1,\psi_\tau(x_2)-\psi_\tau(x_1)\rangle\le C(\tau_0)\|x_2-x_1\|^2\label{eq:psitau-onesidedLip}\\
&\|\psi_\tau(x_2)-\psi_\tau(x_1)\|\le C(\tau_0)\bigl(1+\|x_1\|^3+\|x_2\|^3\bigr)\|x_2-x_1\|,\label{eq:psitau-localLip}
\end{align}
and for all $x\in\R^2$ one has
\begin{equation}\label{eq:psitau-cv}
\|\psi_\tau(x)-F(x)\|\le C(\tau_0)\tau\bigl(1+\|x\|^5\bigr).
\end{equation}
Finally, one has
\begin{equation}\label{eq:psitau0}
\underset{\tau\in(0,\tau_0)}\sup~\|\psi_\tau(0)\|<\infty.
\end{equation}
\end{propo}
The inequality~\eqref{eq:psitau-onesidedLip} states that $\psi_\tau$ satisfies a one-sided Lipschitz continuity property which is uniform with respect to $\tau\in(0,\tau_0)$.
This is similar to the property~\eqref{eq:onesidedF} satisfied by $F$. It is straightforward to check that $\psi_\tau$ is in fact globally Lipschitz continuous
for any fixed $\tau\in(0,\tau_0)$, however this property does not hold uniformly with respect to $\tau\in(0,\tau_0)$. Instead, one has the one-sided Lipschitz continuity property~\eqref{eq:psitau-onesidedLip} and the  local Lipschitz continuity property~\eqref{eq:psitau-localLip} which are both uniform with respect to $\tau\in(0,\tau_0)$.

\section{Preliminary results}\label{sec-prel}
In this section we state and prove several results which are required for the analysis of the three splitting schemes of type~\eqref{eq:LTscheme}. In particular, we give properties of the semigroup (Proposition~\ref{propSemiGrp}), we then prove the properties of the auxiliary mappings $\phi_\tau$ (Proposition~\ref{propo:phitau}) and $\psi_\tau$ (Proposition~\ref{propo:psitau}), and finally we study the well-posedness and moment bounds
for the mild solution of the considered SPDE.

\subsection{Properties of the semigroup}
In this subsection, we study properties of the semigroup generated by the linear operator $\IL$ in the stochastic FitzHugh--Nagumo system~\eqref{eq:SEE}.
In addition, estimates for the operator $(I+\tau\IL)^{-1}$ used in the semi-linear splitting schemes~\eqref{eq:imp-subSPDE} and~\eqref{eq:LTimp} are also provided.

\begin{propo}\label{propSemiGrp}
The semigroup $\bigl(e^{-t\IL}\bigr)_{t\ge 0}$ defined by~\eqref{eq:expoIL} satisfies the following properties:

\noindent $\bullet$ For all $t\ge 0$, $e^{-t\IL}$ is a bounded linear operator from $\HH$ to $\HH$ and from $\EE$ to $\EE$. In addition, for all $t\ge 0$ one has
\begin{equation}\label{eq:propSemiGrp-normes}
\underset{x\in\HH\setminus\{0\}}\sup~\frac{\|e^{-t\IL}x\|_\HH}{\|x\|_\HH}=1,\quad \underset{x\in\EE\setminus\{0\}}\sup~\frac{\|e^{-t\IL}x\|_\EE}{\|x\|_\EE}=1.
\end{equation}

\noindent $\bullet$ Smoothing property. For all $\alpha\in[0,\infty)$, there exists a real number $C_\alpha\in(0,\infty)$ such that, for all $(u,v)\in\HH$ and all $t\in(0,\infty)$, one has
\begin{equation}\label{eq:propSemiGrp-smoothing}
\|e^{-t\IL}\begin{pmatrix}(-\Delta)^{\alpha}u\\ v\end{pmatrix}\|_{\HH}\le C_\alpha\min(1,t)^{-\alpha}\|(u,v)\|_{\HH}.
\end{equation}
\noindent $\bullet$ Temporal regularity.  For all $\mu,\nu\ge 0$ with $\mu+\nu\le 1$, there exists a real number $C_{\mu,\nu}\in(0,\infty)$ such that, for all $x=(u,v)\in \HH^{2\nu}$ and all $t_1,t_2\in(0,\infty)$, one has
\begin{equation}\label{eq:propSemiGrp-regul}
\|e^{-t_2\IL}x-e^{-t_1\IL}x\|_{\HH}\le C_{\mu,\nu}\frac{|t_2-t_1|^{\mu+\nu}}{\min(t_2,t_1)^\mu}\|(-\Delta)^\nu u\|_H.
\end{equation}
\end{propo}

\begin{proof}
$\bullet$ On the one hand, since the eigenvalues $\bigl(\lambda_j\bigr)_{j\ge 0}$ of $-\Delta$ are nonnegative, it is straightforward to see that for all $x=(u,v)\in \HH$ and $t\ge 0$ one has $e^{-t\IL}x\in\HH$, and
\[
\|e^{-t\IL}x\|_\HH^2=\|e^{t\Delta}u\|_H^2+\|v\|_H^2\le \|u\|_{H}^2+\|v\|_{H}^2=\|x\|_{\HH}^2.
\]
This proves that $e^{-t\IL}$ is a bounded linear operator from $\HH$ to $\HH$ for all $t\ge 0$, and that
\[
\underset{x\in\HH\setminus\{0\}}\sup~\frac{\|e^{-t\IL}x\|_\HH}{\|x\|_\HH}\le 1.
\]
On the other hand, using the formula for the Green function of the heat equation with homogeneous Neumann boundary conditions,
the semigroup $\bigl(e^{t\Delta}\bigr)_{t\ge 0}$ defined by~\eqref{eq:expoDelta} satisfies the following properties:
for all $t\ge 0$ and $u\in E$,
one has $e^{t\Delta}u\in E$ and $\|e^{t\Delta}u\|_{E}\le \|u\|_E$. As a consequence, for all $x=(u,v)\in \EE$, one has $e^{-t\IL}x=(e^{t\Delta}u,v)\in\EE$ and
\[
\|e^{-t\IL}x\|_\EE=\max\bigl(\|e^{t\Delta}u\|_E,\|v\|_E\bigr)\le \max\bigl(\|u\|_E,\|v\|_E\bigr)=\|x\|_\EE.
\]
To conclude the proof of~\eqref{eq:propSemiGrp-normes}, it suffices to check that for $x=(0,v)$ and all $t\ge 0$ one has $e^{-t\IL}x=x$.

$\bullet$ The smoothing property~\eqref{eq:propSemiGrp-smoothing} is a straightforward consequence of the smoothing property for the semigroup $\bigl(e^{t\Delta}\bigr)_{t\ge 0}$: for all $\alpha\in[0,\infty)$, $t\ge 0$ and $u\in H$,
one has (recall that $\lambda_0=0$)
\[
\|e^{t\Delta}(-\Delta)^\alpha u\|_H^2=\sum_{j\ge 1}e^{-2t\lambda_j}\lambda_j^{2\alpha}\langle u,e_j\rangle_H^2\le \underset{\xi\in(0,\infty)}\sup~\bigl(\xi^{2\alpha}e^{-2\xi}\bigr)~t^{-2\alpha}\|u\|_H^2.
\]
As a consequence, for all $\alpha\in[0,\infty)$, $t\ge 0$ and $x=(u,v)\in\HH$, one has
\[
\|e^{-t\IL}\begin{pmatrix}(-\Delta)^{\alpha}u\\ v\end{pmatrix}\|_{\HH}^2=\|e^{t\Delta}(-\Delta)^{\alpha}u\|_H^2+\|v\|_H^2\le C_\alpha^2 t^{-2\alpha}\|u\|_H^2+\|v\|_H^2\le C_\alpha^2 \min(1,t)^{-2\alpha}\|x\|_{\HH}^2.
\]

$\bullet$ The regularity property~\eqref{eq:propSemiGrp-regul} is a straightforward consequence of the following regularity property for the semigroup $\bigl(e^{t\Delta}\bigr)_{t\ge 0}$: for all $\mu,\nu\in[0,1]$ with $\mu+\nu\le 1$, $0\le t_1\le t_2$ and $u\in H^{2\nu}$, one has
\begin{align*}
\|e^{t_2\Delta}u-e^{t_1\Delta}u\|^2_H&=\|(e^{(t_2-t_1)\Delta}-I)e^{t_1\Delta}u\|^2_H\\
&=\sum_{j\ge 1}\bigl(e^{-(t_2-t_1)\lambda_j}-1\bigr)^2e^{-2t_1\lambda_j}\langle u,e_j\rangle_H^2\\
&\le 2^{2(\mu+\nu)}(t_2-t_1)^{2(\mu+\nu)}\sum_{j\ge 1}\lambda_j^{2(\mu+\nu)}e^{-2t_1\lambda_j}\langle u,e_j\rangle_H^2\\
&\le 2^{2(\mu+\nu)}\underset{\xi\in(0,\infty)}\sup~\bigl(\xi^{2\mu}e^{-2\xi}\bigr)~\frac{(t_2-t_1)^{2(\mu+\nu)}}{t_1^{2\mu}}\sum_{j\ge 1}\lambda_j^{2\nu}\langle u,e_j\rangle_H^2\\
&\le 2^{2(\mu+\nu)}\underset{\xi\in(0,\infty)}\sup~\bigl(\xi^{2\mu}e^{-2\xi}\bigr)~\frac{(t_2-t_1)^{2(\mu+\nu)}}{t_1^{2\mu}}\|(-\Delta)^\nu u\|_H^2.
\end{align*}
As a consequence, for all $\mu,\nu\in[0,1]$ with $\mu+\nu\le 1$, $0\le t_1\le t_2$ and $x=(u,v)\in H^{2\nu}\times H$, one has
\[
\|e^{-t_2\IL}x-e^{-t_1\IL}x\|_{\mathcal{H}}=\|e^{t_2\Delta}u-e^{t_1\Delta}u\|_H\le C_{\mu,\nu}
\frac{|t_2-t_1|^{\mu+\nu}}{t_1^\mu}\|(-\Delta)^\nu u\|_H.
\]
The proof of Proposition~\ref{propSemiGrp} is thus completed.
\end{proof}

In the sequel, the following properties are also used for the analysis of the splitting scheme~\eqref{eq:LTimp} for which a linear implicit Euler method is used for the approximation~\eqref{eq:imp-subSPDE} of the stochastic convolution: for all $t\ge 0$, $(I+t\IL)^{-1}$ is a bounded linear operator from $\HH$ to $\HH$ and from $\EE$ to $\EE$, and one has
\begin{equation}\label{eq:SemiGrp-normes-imp}
\underset{x\in\HH\setminus\{0\}}\sup~\frac{\|(I+t\IL)^{-1}x\|_\HH}{\|x\|_\HH}=1,\quad \underset{x\in\EE\setminus\{0\}}\sup~\frac{\|(I+t\IL)^{-1}x\|_\EE}{\|x\|_\EE}=1.
\end{equation}
The proof of the inequality~\eqref{eq:SemiGrp-normes-imp} is straightforward. Indeed, for all $x\in\HH$ or $x\in\EE$, and all $t\ge 0$, one has
\[
(I+t\IL)^{-1}x=\int_0^\infty e^{-(I+t\IL)s}x\,\dd s.
\]
Using~\eqref{eq:propSemiGrp-normes}, one then obtains the inequalities
\begin{align*}
&\|(I+t\IL)^{-1}x\|_\HH\le \int_{0}^{\infty}e^{-s}\|e^{-ts\IL}x\|_\HH\,\dd s\le \int_{0}^{\infty}e^{-s}\,\dd s\|x\|_\HH=\|x\|_\HH\\
&\|(I+t\IL)^{-1}x\|_\EE\le \int_{0}^{\infty}e^{-s}\|e^{-ts\IL}x\|_\EE\,\dd s\le \int_{0}^{\infty}e^{-s}\,\dd s\|x\|_\EE=\|x\|_\EE.
\end{align*}
Like in the proof of~\eqref{eq:propSemiGrp-normes}, choosing $x=(0,v)$ gives $(I+t\IL)^{-1}x=x$ for all $t\ge 0$, and thus concludes the proof of~\eqref{eq:SemiGrp-normes-imp}.

\subsection{Proofs of Propositions~\ref{propo:phitau} and~\ref{propo:psitau}}\label{cestlaquonprouve}

In order to prove Propositions~\ref{propo:phitau} and~\ref{propo:psitau} which state properties of the mappings $\phi_\tau:\R^2\to\R^2$ and $\psi_\tau:\R^2\to\R^2$
defined by~\eqref{eq:phitau}~and~\eqref{eq:psitau}, it is convenient to introduce the auxiliary mappings $\phi_t^{\AC}:\R\to\R$ and $\psi_t^{\AC}:\R\to\R$, defined as follows: for all $t\in(0,\infty)$ and $u\in\R$, set
\begin{equation}\label{eq:phipsiAC}
\phi_t^{\AC}(u)=\frac{u}{\sqrt{u^2+(1-u^2)e^{-2t}}},\quad \psi_t^{\AC}(u)=\frac{\phi_t^{\AC}(u)-u}{t}.
\end{equation}
The mapping $\phi_t^{\AC}$ is the flow map associated with the nonlinear differential equation, see the subsystem~\eqref{eq:subODE-NL},
\[
\frac{\dd u^{\AC}(t)}{\dd t}=u^{\AC}(t)-(u^{\AC}(t))^3,
\]
meaning that $u^{\AC}(t)=\phi_t^{\AC}(u^{\AC}(0))$ for all $t\ge 0$. The properties of the mappings $\phi_\tau^{\AC}$ and $\psi_\tau^{\AC}$ stated in Lemma~\ref{lem:phipsiAC} are given by~\cite[Lemma~3.1--3.4]{MR3986273}.
\begin{lemma}\label{lem:phipsiAC}
There exists $C(\tau_0)\in(0,\infty)$ such that for all $\tau\in(0,\tau_0)$, the mappings $\phi_\tau^{\AC}:\R\to\R$ and $\psi_\tau^{\AC}:\R\to\R$ satisfy the following properties:

$\bullet$ For all $\tau\in(0,\tau_0)$ and $u_1,u_2\in\R$, one has
\begin{equation}\label{eq:lemphipsiAC-Lip-phi}
|\phi_\tau^{\AC}(u_2)-\phi_\tau^{\AC}(u_1)|\le e^{\tau}|u_2-u_1|.
\end{equation}
$\bullet$ For all $\tau\in(0,\tau_0)$ and $u_1,u_2\in\R$, one has
\begin{align}
&\bigl(u_2-u_1\bigr)\bigl(\psi_\tau^{\AC}(u_2)-\psi_\tau^{\AC}(u_1)\bigr)\le C(\tau_0)|u_2-u_1|^2, \label{eq:lemphipsiAC-onesidedLip-psi}\\
&|\psi_\tau^{\AC}(u_2)-\psi_\tau^{\AC}(u_1)|\le C(\tau_0)\bigl(1+|u_1|^3+|u_2|^3\bigr)|u_2-u_1|,\label{eq:lemphipsiAC-localLip-psi}
\end{align}
and for all $\tau\in(0,\tau_0)$ and $u\in\R$, one has
\begin{equation}\label{eq:lemphipsiAC-cv-psi}
|\psi_\tau(u)-(u-u^3)|\le C(\tau_0)\tau\bigl(1+|u|^5\bigr).
\end{equation}
\end{lemma}

We are now in position to prove Proposition~\ref{propo:phitau}. The result is straightforward: $\phi_\tau$ is the composition of the two globally Lipschitz continuous mappings $\phi_\tau^{\rm L}$ and $\phi_\tau^{\rm NL}$. The proof is given to exhibit the dependence of the Lipschitz constant with respect to the time-step size $\tau\in(0,\tau_0)$.
\begin{proof}[Proof of Proposition~\ref{propo:phitau}]
Note that for all $\tau\in(0,\tau_0)$ and $x=(u,v)\in\R^2$ one has
\[
\phi_\tau^{\rm NL}(x)=\begin{pmatrix}\phi_\tau^{\AC}(u)\\ v+\beta \tau\end{pmatrix}.
\]
Using the definition~\eqref{eq:phitau} and the inequality~\eqref{eq:lemphipsiAC-Lip-phi} from Lemma~\ref{lem:phipsiAC}, one then obtains the following inequality: for all $\tau\in(0,\tau_0)$ and all $x_1=(u_1,v_1),x_2=(u_2,v_2)\in\R^2$, one has
\begin{align*}
\|\phi_\tau(x_2)-\phi_\tau(x_{1})\|^2&=\|\phi_\tau^{\rm L}(\phi_\tau^{\rm NL}(x_2))-\phi_\tau^{\rm L}(\phi_\tau^{\rm NL}(x_1))\|^2\\
&=\|e^{\tau B}\bigl(\phi_\tau^{\rm NL}(x_2)-\phi_\tau^{\rm NL}(x_1)\bigr)\|^2\\
&\le e^{2\tau\vvvert B\vvvert}\|\phi_\tau^{\rm NL}(x_2)-\phi_\tau^{\rm NL}(x_1)\|^2\\
&\le e^{2\tau\vvvert B\vvvert}\bigl(|\phi_\tau^{\AC}(u_2)-\phi_\tau^{\AC}(u_1)|^2+|v_2-v_1|^2\bigr)\\
&\le e^{2\tau\vvvert B\vvvert}\bigl(e^{2\tau}|u_2-u_1|^2+|v_2-v_1|^2\bigr)\\
&\le e^{2\tau(1+\vvvert B\vvvert)}\|x_2-x_1\|^2.
\end{align*}
This concludes the proof of Proposition~\ref{propo:phitau}.
\end{proof}

In order to prove Proposition~\ref{propo:psitau}, the main tool is the following expression for the mapping $\psi_\tau$ defined by~\eqref{eq:psitau}: for all $\tau\in(0,\tau_0)$ and $x\in\R^2$, one has
\begin{equation}\label{eq:identitypsitau}
\psi_\tau(x)=\psi_\tau^{\rm L}(\phi_\tau^{\rm NL}(x))+\psi_\tau^{\rm NL}(x),
\end{equation}
where the mappings $\psi_\tau^{\rm L}$ and $\psi_\tau^{\rm NL}$ are given by
\begin{align*}
&\psi_\tau^{\rm L}(x)=\frac{\phi_\tau^{\rm L}(x)-x}{\tau}=\frac{e^{\tau B}-I}{\tau}x\\
&\psi_\tau^{\rm NL}(x)=\frac{\phi_\tau^{\rm NL}(x)-x}{\tau}=\begin{pmatrix} \psi_\tau^{\AC}(u)\\ \beta\end{pmatrix}
\end{align*}
for all $\tau\in(0,\tau_0)$ and $x=(u,v)\in\R^2$.

The proof of the equality~\eqref{eq:identitypsitau} is straightforward: using~\eqref{eq:phitau}, one has
\begin{align*}
\psi_\tau(x)&=\frac{\phi_\tau(x)-x}{\tau}=\frac{\phi_\tau^{\rm L}(\phi_\tau^{\rm NL}(x))-\phi_\tau^{\rm NL}(x)}{\tau}+\frac{\phi_\tau^{\rm NL}(x)-x}{\tau}\\
&=\psi_\tau^{\rm L}(\phi_\tau^{\rm NL}(x))+\psi_\tau^{\rm NL}(x).
\end{align*}

Having the identity~\eqref{eq:identitypsitau} at hand, we are now in position to prove Proposition~\ref{propo:psitau}.
\begin{proof}[Proof of Proposition~\ref{propo:psitau}]
Note that the mapping $\psi_\tau^{\rm L}:\R^2\to\R^2$ is linear and therefore is globally Lipschitz continuous. In addition, for all $\tau\in(0,\tau_0)$ and $x_1,x_2\in\R^2$, one has
\begin{equation}\label{eq:psitauL-Lip}
\|\psi_\tau^{\rm L}(x_2)-\psi_\tau^{\rm L}(x_1)\|\le \vvvert \frac{e^{\tau B}-I}{\tau}\vvvert \|x_2-x_1\|\le \frac{e^{\tau_0 \vvvert B\vvvert}-1}{\tau_0}\|x_2-x_1\|,
\end{equation}
using the inequalities
\begin{align*}
\vvvert \frac{e^{\tau B}-I}{\tau}\vvvert&=\vvvert \sum_{k=1}^{\infty}\frac{\tau^{k-1}}{k!}B^k\vvvert\le \sum_{k=1}^{\infty}\frac{\tau^{k-1}}{k!}\vvvert B\vvvert^k\le \sum_{k=1}^{\infty}\frac{\tau_0^{k-1}}{k!}\vvvert B\vvvert^k=\frac{e^{\tau_0 \vvvert B\vvvert}-1}{\tau_0}.
\end{align*}
Let us first prove the one-sided Lipschitz continuity property~\eqref{eq:psitau-onesidedLip}: for all $\tau\in(0,\tau_0)$ and $x_1,x_2\in\R^2$, using the identity~\eqref{eq:identitypsitau}, then the Cauchy--Schwarz inequality and~\eqref{eq:psitauL-Lip}, one has
\begin{align*}
\langle x_2-x_1,\psi_\tau(x_2)-\psi_\tau(x_1)\rangle&=\langle x_2-x_1,\psi_\tau^{\rm L}(\phi_\tau^{\rm NL}(x_2))-\psi_\tau^{\rm L}(\phi_\tau^{\rm NL}(x_1))\rangle\\
&+\langle x_2-x_1,\psi_\tau^{\rm NL}(x_2)-\psi_\tau^{\rm NL}(x_1)\rangle\\
&\le \frac{e^{\tau_0 \vvvert B\vvvert}-1}{\tau_0} \|x_2-x_1\| \|\phi_\tau^{\rm NL}(x_2)-\phi_\tau^{\rm NL}(x_2)\|\\
&+\langle x_2-x_1,\psi_\tau^{\rm NL}(x_2)-\psi_\tau^{\rm NL}(x_1)\rangle.
\end{align*}
On the one hand, using the same arguments as in the proof of Proposition~\ref{propo:phitau}, one has
\[
\|\phi_\tau^{\rm NL}(x_2)-\phi_\tau^{\rm NL}(x_1)\|\le e^\tau \|x_2-x_1\|\le e^{\tau_0} \|x_2-x_1\|.
\]
On the other hand, for all $x=(u,v)\in\R^2$ one has
\[
\psi_\tau^{\rm NL}(x)=\begin{pmatrix} \psi_\tau^{\AC}(u)\\ \beta\end{pmatrix}.
\]
Using the inequality~\eqref{eq:lemphipsiAC-onesidedLip-psi} from Lemma~\ref{lem:phipsiAC}, one then obtains
\[
\langle x_2-x_1,\psi_\tau^{\rm NL}(x_2)-\psi_\tau^{\rm NL}(x_1)\rangle\le e^\tau\|x_2-x_1\|^2\le e^{\tau_0}\|x_2-x_1\|^2.
\]
Gathering the results then gives
\[
\langle x_2-x_1,\psi_\tau(x_2)-\psi_\tau(x_1)\rangle\le \Bigl(\frac{e^{\tau_0 \vvvert B\vvvert}-1}{\tau_0}+1\Bigr)e^{\tau_0}\|x_2-x_1\|^2,
\]
which concludes the proof of the inequality~\eqref{eq:psitau-onesidedLip}.

Let us now prove the local Lipschitz continuity property~\eqref{eq:psitau-localLip}. Using the identity~\eqref{eq:identitypsitau} and the inequality~\eqref{eq:lemphipsiAC-localLip-psi}, for all $\tau\in(0,\tau_0)$ and $x_1=(u_1,v_1),x_2=(u_2,v_2)\in\R^2$, one has
\begin{align*}
\|\psi_\tau(x_2)-\psi_\tau(x_1)\|&\le \|\psi_\tau^{\rm L}(\phi_\tau^{\rm NL}(x_2))-\psi_\tau^{\rm L}(\phi_\tau^{\rm NL}(x_1))\|+\|\psi_\tau^{\rm NL}(x_2)-\psi_\tau^{\rm NL}(x_1)\|\\
&\le \frac{e^{\tau_0 \vvvert B\vvvert}-1}{\tau_0}\|\phi_\tau^{\rm NL}(x_2)-\phi_\tau^{\rm NL}(x_1)\|+C(\tau_0)\bigl(1+|u_1|^3+|u_2|^3\bigr)|u_2-u_1|\\
&\le \Bigl(\frac{e^{\tau_0 \vvvert B\vvvert}-1}{\tau_0}e^{\tau_0}+C(\tau_0)\Bigr)\bigl(1+\|x_1\|^3+\|x_2\|^3\bigr)\|x_2-x_1\|.
\end{align*}
Let us now prove the error estimate~\eqref{eq:psitau-cv}. Using the identities~\eqref{eq:F} and~\eqref{eq:identitypsitau}, for all $\tau\in(0,\tau_0)$ and $x=(u,v)\in\R^2$, one has
\[
\|\psi_\tau(x)-F(x)\|\le \|\psi_\tau^{\rm L}(\phi_\tau^{\rm NL}(x))-F^{\rm L}(x)\|+\|\psi_\tau^{\rm NL}(x)-F^{\rm NL}(x)\|.
\]
On the one hand, using the inequality~\eqref{eq:psitauL-Lip}, the expressions of the linear mappings $F^{\rm L}$ and $\psi_\tau^{\rm L}$ and the definition of $\psi_\tau^{\rm NL}$, one has
\begin{align*}
\|\psi_\tau^{\rm L}(\phi_\tau^{\rm NL}(x))-F^{\rm L}(x)\|&\le \|\psi_\tau^{\rm L}(\phi_\tau^{\rm NL}(x))-F^{\rm L}(\phi_\tau^{\rm NL}(x))\|+\|F^{\rm L}(\phi_\tau^{\rm NL}(x))-F^{\rm L}(x)\|\\
&\le \vvvert \frac{e^{\tau B}-I-\tau B}{\tau}\vvvert\|\phi_\tau^{\rm NL}(x)\|+\tau\vvvert B\vvvert\|\psi_\tau^{\rm NL}(x)\|.
\end{align*}
Note that $\phi_\tau^{\rm NL}(0)=(\phi_\tau^{\AC}(0),\beta\tau)=(0,\beta\tau)$ and $\psi_\tau^{\rm NL}(0)=(\psi_\tau^{\AC}(0),\beta)=(0,\beta)$. In addition, one has
\[
\vvvert \frac{e^{\tau B}-I-\tau B}{\tau}\vvvert\le\sum_{k=2}^{\infty}\frac{\tau^{k-1}}{k!}\vvvert B\vvvert^k\le \tau \sum_{k=2}^{\infty}\frac{\tau_0^{k-2}}{k!}\vvvert B\vvvert^k=\tau\frac{e^{\tau_0\vvvert B\vvvert}-1-\tau_0\vvvert B\vvvert}{\tau_0}.
\]
Therefore, using the inequalities~\eqref{eq:phitau-Lip} from Proposition~\ref{propo:phitau} and~\eqref{eq:lemphipsiAC-localLip-psi} from Lemma~\ref{lem:phipsiAC}, one has
\[
\|\psi_\tau^{\rm L}(\phi_\tau^{\rm NL}(x))-F^{\rm L}(x)\|\le C(\tau_0)\tau(1+\|x\|^4).
\]

On the other hand, using the inequality~\eqref{eq:lemphipsiAC-cv-psi} from Lemma~\ref{lem:phipsiAC}, one has
\[
\|\psi_\tau^{\rm NL}(x)-F^{\rm NL}(x)\|=|\psi_\tau^{\AC}(u)-(u-u^3)|\le C(\tau_0)\tau\bigl(1+|u|^5\bigr).
\]
Gathering the estimates then gives the inequality
\[
\|\psi_\tau(x)-F(x)\|\le C(\tau_0)\tau(1+\|x\|^5),
\]
which concludes the proof of~\eqref{eq:psitau-cv}.

It remains to prove the inequality~\eqref{eq:psitau0}. The proof is straightforward: using~\eqref{eq:identitypsitau} and the equalities
$\phi_\tau^{\AC}(0)=\psi_\tau^{\AC}(0)=0$, one has
\[
\psi_\tau(0)=\frac{e^{\tau B}-I}{\tau}\phi_\tau^{\rm NL}(0)+\psi_\tau^{\rm NL}(0)=e^{\tau B}\begin{pmatrix} 0\\ \beta\end{pmatrix}.
\]
Therefore one gets
\[
\underset{\tau\in(0,\tau_0)}\sup~\|\psi_\tau(0)\|\le e^{\tau_0\vvvert B\vvvert}|\beta|.
\]
The proof of Proposition~\ref{propo:psitau} is thus completed.
\end{proof}

Let us conclude this subsection with a remark concerning the order of the composition of the two subsystems to define the splitting schemes, see equation~\eqref{eq:phitau}.
\begin{rem}\label{rem:schemehat}
Let $\hat{\phi}_\tau:\R^2\to\R^2$ be defined as follows: for all $\tau\in(0,\tau_0)$, set
\begin{equation}\label{eq:phitauhat}
\hat{\phi}_\tau=\phi_\tau^{\rm NL}\circ\phi_\tau^{\rm L}.
\end{equation}
Compared with the definition~\eqref{eq:phitau} of $\phi_\tau$, the order of the composition of the integrators $\phi_\tau^{\rm L}$ and $\phi_\tau^{\rm NL}$ associated with the subsystems~\eqref{eq:subODE-L} and~\eqref{eq:subODE-NL} respectively is reversed. Define also
\begin{equation}\label{eq:psitauhat}
\hat{\psi}_\tau(x)=\frac{\hat{\phi}_\tau(x)-x}{\tau}
\end{equation}
for all $\tau\in(0,\tau_0)$ and $x\in\R^2$. Using the mapping $\hat{\phi}_\tau$, modifying the definition of the scheme~\eqref{eq:LTscheme} gives the alternative splitting scheme
\begin{equation}\label{eq:LTschemehat}
\hat{X}_{n+1}=\IA_\tau\hat{\phi}_\tau(\hat{X}_n)+\int_{t_n}^{t_{n+1}}\IB_{t_{n+1}-s}\,\dd\WW(s)
\end{equation}
for the approximation of the stochastic evolution equation~\eqref{eq:SEE}. Precisely, alternatives of the splitting schemes~\eqref{eq:LTexact},~\eqref{eq:LTexpo} and~\eqref{eq:LTimp} are obtained from the formulation~\eqref{eq:LTschemehat}. However, the analysis performed in this paper does not encompass the case of the scheme~\eqref{eq:LTschemehat}, due to missing properties for the mapping $\hat{\psi}_\tau$, compared with $\psi_\tau$, as explained below.

Note that the result of Proposition~\ref{propo:phitau} also holds with $\phi_\tau$ replaced by $\hat{\phi}_\tau$. However, it is not clear whether the one-sided Lipschitz continuity property~\eqref{eq:psitau-onesidedLip} from Proposition~\ref{propo:psitau} holds also with $\psi_\tau$ replaced by $\hat{\psi}_\tau$ (uniformly with respect to $\tau\in(0,\tau_0)$). The proof of the inequality~\eqref{eq:psitau-onesidedLip} exploits the global Lipschitz continuity property~\eqref{eq:psitauL-Lip} of the auxiliary mapping $\psi_\tau^{\rm L}$, which is a linear mapping from $\R^2$ to $\R^2$. Instead of the identity~\eqref{eq:identitypsitau}, one has
\begin{equation}\label{eq:identitypsitauhat}
\hat{\psi}_\tau(x)=\psi_\tau^{\rm NL}(\phi_\tau^{\rm L}(x))+\psi_\tau^{\rm L}(x),
\end{equation}
and since $\psi_\tau^{\rm NL}$ is not globally Lipschitz continuous uniformly with respect to $\tau\in(0,\tau_0)$, the arguments of the proof above cannot be repeated for the splitting scheme~\eqref{eq:LTschemehat}.
\end{rem}

\subsection{Moment bounds for the solutions of the stochastic evolution equations~\eqref{eq:SEE} and~\eqref{eq:SEEaux}}

Let us first state the moment bounds for the stochastic convolution defined by~\eqref{eq:Z}.
\begin{lemma}\label{lem:Zbound}
Let $\bigl(\mathcal{Z}(t)\bigr)_{t\ge 0}$ be defined by~\eqref{eq:Z}. For all $T\in(0,\infty)$ and $p\in[1,\infty)$, one has
\[
\underset{0\le t\le T}\sup~\E[\|\mathcal{Z}(t)\|_\EE^p]<\infty.
\]
\end{lemma}

\begin{proof}
Let us only provide the sketch of the proof. To deal with homogeneous Neumann boundary conditions,
it is convenient to introduce $Z_0(t)=\langle Z(t),e_0\rangle e_0=\beta_0(t)e_0$ and ${Z}_\perp(t)=Z(t)-Z_0(t)$ for all $t\ge 0$. Let also $\mathcal{Z}_0(t)=\begin{pmatrix} Z_0(t)\\0\end{pmatrix}$ and ${\mathcal{Z}}_\perp(t)=\mathcal{Z}(t)-\mathcal{Z}_0(t)$. On the one hand, one has
\[
\underset{0\le t\le T}\sup~\E[\|\mathcal{Z}_0(t)\|_\EE^p]=\underset{0\le t\le T}\sup~\E[\|Z_0(t)\|_E^p]\le \underset{0\le t\le T}\sup~\E[|\beta_0(t)|^p]\|e_0\|_E^p\le CT^{\frac{p}{2}}.
\]
On the other hand, applying the temporal and spatial increment bounds~\cite[Lemma~5.21]{MR3236753} and the Kolmogorov regularity criterion~\cite[Theorem~C.6]{MR3222416} gives
\[
\underset{0\le t\le T}\sup~\E[\|\mathcal{Z}_\perp(t)\|_\EE^p]=\underset{0\le t\le T}\sup~\E[\|Z_\perp(t)\|_E^p]\le C(T)<\infty.
\]
Combining the moment bounds for $\mathcal{Z}_0(t)$ and $\mathcal{Z}_\perp(t)$ then concludes the proof of Lemma~\ref{lem:Zbound}.
\end{proof}

We now state well-posedness and moment bounds properties, first for the solutions to the stochastic FitzHugh--Nagumo SPDE system~\eqref{eq:FhNsystem}, 
second for the solutions to the auxiliary SPDE~\eqref{eq:SEE}.
\begin{propo}\label{propo:SEE}
For any initial value $x_0\in\HH$, the stochastic evolution equation~\eqref{eq:SEE} admits a unique global mild solution $\bigl(X(t)\bigr)_{t\ge 0}$, in the sense that~\eqref{eq:SEEmild} is satisfied. Moreover, for all $T\in(0,\infty)$ and all $p\in[1,\infty)$, there exists $C_p(T)\in(0,\infty)$ such that for all $x_0\in\EE$ one has
\begin{equation}\label{eq:momentSEE}
\underset{0\le t\le T}\sup~\E[\|X(t)\|_\EE^p]\le C_p(T)\bigl(1+\|x_0\|_\EE^p\bigr).
\end{equation}
\end{propo}

\begin{propo}\label{propo:SEEaux}
For any initial value $x_0\in\HH$ and for all $\tau\in(0,\tau_0)$, the stochastic evolution equation~\eqref{eq:SEEaux} admits a unique global mild solution $\bigl(X_\tau(t)\bigr)_{t\ge 0}$, in the sense that
\begin{equation}\label{eq:SEEauxmild}
X_\tau(t)=e^{-t\IL}x_0+\int_0^t e^{-(t-s)\IL}\psi_\tau(X_\tau(s))\,\dd s+\int_0^t e^{-(t-s)\IL}\,\dd \WW(s)
\end{equation}
is satisfied for all $t\ge 0$. Moreover, for all $T\in(0,\infty)$ and all $p\in[1,\infty)$, there exists $C_p(T,\tau_0)\in(0,\infty)$ such that for all $x_0\in\EE$ one has
\begin{equation}\label{eq:momentSEEaux}
\underset{\tau\in(0,\tau_0)}\sup~\underset{0\le t\le T}\sup~\E[\|X_\tau(t)\|_\EE^p]\le C_p(T)\bigl(1+\|x_0\|_\EE^p\bigr).
\end{equation}
\end{propo}

The detailed proofs of Propositions~\ref{propo:SEE} and~\ref{propo:SEEaux} are omitted.
However let us emphasize that the main arguments used in the proofs are, on the one hand,
the one-sided Lipschitz continuity properties~\eqref{eq:onesidedF} and~\eqref{eq:psitau-onesidedLip} of $F$ and $\psi_\tau$ respectively, and on the other hand, the moment bounds on $\mathcal{Z}(t)$ from Lemma~\ref{lem:Zbound}. Observe that the mapping $\psi_\tau$ is globally Lipschitz continuous for any $\tau>0$, therefore the existence and uniqueness of the mild solution $\bigl(X_\tau(t)\bigr)_{t\ge 0}$ satisfying~\eqref{eq:SEEauxmild} follows from standard fixed point arguments, see for instance~\cite[Theorem~7.5]{MR3236753}. The proof of the moment bounds~\eqref{eq:momentSEEaux} requires some care: indeed, one needs to obtain upper bounds which are uniform with respect to $\tau\in(0,\tau_0)$, and applying~\cite[Theorem~7.5]{MR3236753} would not be appropriate since the Lipschitz constant of $\psi_\tau$ is unbounded for $\tau\in(0,\tau_0)$. Introducing $Y_\tau(t)=X_\tau(t)-\mathcal{Z}(t)$, one obtains the moment bounds~\eqref{eq:momentSEEaux} using the one-sided Lipschitz continuity property~\eqref{eq:psitau-onesidedLip} from Proposition~\ref{propo:psitau}, which is uniform with respect to $\tau\in(0,\tau_0)$. Similar arguments are used to prove Proposition~\ref{propo:SEE}. Propositions~\ref{propo:SEE} and~\ref{propo:SEEaux} are variants of~\cite[Propositions~1 and~2]{MR3986273} for the analysis of the stochastic Allen--Cahn equation and we refer to~\cite[Proposition~6.2.2]{MR1840644} for a more general version. Some arguments need to be adapted since the considered systems~\eqref{eq:SEE} and~\eqref{eq:SEEaux} are not parabolic systems.

Finally, let us state the following result which is required in Section~\ref{sectionProofs} below.
\begin{lemma}\label{lem:regularityXtau}
For all $T\in(0,\infty)$, $p\in[1,\infty)$ and $\alpha\in[0,\frac14)$, there exists $C_{\alpha,p}(T)\in(0,\infty)$ such that for all $x_0=(u_0,v_0)\in\HH^{2\alpha}\cap\EE$, all $\tau\in(0,\tau_0)$ and $t_1,t_2\in[0,T]$, one has
\begin{equation}\label{eq:regularityXtau}
\bigl(\E[\|X_\tau(t_2)-X_\tau(t_1)\|_\HH^p]\bigr)^{\frac1p}\le C_{\alpha,p}(T)|t_2-t_1|^{\alpha}\bigl(1+\|(-\Delta)^\alpha u_0\|_H^4+\|x_0\|_\EE^4\bigr).
\end{equation}
\end{lemma}
\begin{proof}

Let $0\leq t_1<t_2\leq T$, using the mild form~\eqref{eq:SEEauxmild} of the auxiliary stochastic evolution equation, we obtain the estimate
\begin{align*}
\bigl(\E\left[ \|X_\tau(t_2)-X_\tau(t_1)\|_\HH^p \right]\bigr)^{\frac1p}&\leq
\|e^{-t_2\IL}x_0-e^{-t_1\IL}x_0\|_{\HH}+
\bigl(\E\left[ \|\mathcal{Z}(t_2)-\mathcal{Z}(t_1)\|_\HH^p \right]\bigr)^{\frac1p} \\
&\quad+\int_0^{t_1}\bigl(\E\left[ \|\left( e^{-(t_2-s)\IL}-e^{-(t_1-s)\IL} \right)\psi_\tau(X_\tau(s))\|_\HH^p \right]
\bigr)^{\frac1p}\,\dd s\\
&\quad+\int_{t_1}^{t_2}\bigl(\E\left[ \|e^{-(t_2-s)\IL}\psi_\tau(X_\tau(s))\|_\HH^p \right]\bigr)^{\frac1p}\,\dd s,
\end{align*}
where we recall that $\mathcal{Z}(t)$ denotes the stochastic convolution~\eqref{eq:Z}.

The first term on the right-hand side is estimated using the inequality~\eqref{eq:propSemiGrp-regul} in order
to get
$$
\|e^{-t_2\IL}x_0-e^{-t_1\IL}x_0\|_\HH\leq|t_2-t_1|^\alpha\|(-\IL)^\alpha x_0\|_\HH.
$$
The second term corresponds to the temporal regularity of the stochastic convolution
$$
\bigl(\E\left[ \|\mathcal{Z}(t_2)-\mathcal{Z}(t_1)\|_\HH^p \right]\bigr)^{\frac1p}\leq|t_2-t_1|^\alpha.
$$
This is obtained combining the proofs of Lemma~\ref{lem:Zbound} and of \cite[Theorem~4.4]{brehierhal}.

The last two terms are estimated using the polynomial growth $\|\psi_\tau(x)\|\leq C(\tau_0)\left(1+\|x\|\right)^4$,
see equations~\eqref{eq:psitau-localLip}~and~\eqref{eq:psitau0} in Proposition~\ref{propo:psitau}.
Indeed, one has
\begin{align*}
\|\left( e^{-(t_2-s)\IL}-e^{-(t_1-s)\IL} \right)\psi_\tau(X_\tau(s))\|_\HH
&\leq C_\alpha
\frac{|t_2-t_1|^\alpha}{|t_1-s|^\alpha}\| \psi_\tau(X_\tau(s)) \|_{\HH}\\
&\leq C_\alpha(\tau_0)
\frac{|t_2-t_1|^\alpha}{|t_1-s|^\alpha}\left(1+\| X_\tau(s)\|^4_{\EE}\right)
\end{align*}
and
$$
\|e^{-(t_2-s)\IL}\psi_\tau(X_\tau(s))\|_\HH\leq \left(1+\| X_\tau(s)\|^4_{\EE}\right)
$$
for the last term. One concludes the proof using the moment bounds of the solution of the auxiliary stochastic evolution equation, see Proposition~\ref{propo:SEEaux}.
\end{proof}

\section{Proofs of the main results}\label{sectionProofs}
In this section, we provide the detailed proofs for the main results of the present work. We start by proving moment bounds for the three splitting schemes (Theorem~\ref{theo:momentbounds}). We then prove the strong error estimates with rate of convergence at least $1/4$ (Theorem~\ref{theo:error}).

\subsection{Proof of Theorem~\ref{theo:momentbounds}}

The proof of the moment bounds~\eqref{eq:momentboundsLTscheme} given below is inspired by the proof of~\cite[Proposition~3]{MR3986273}
and requires some auxiliary tools.

Given the time-step size $\tau\in(0,\tau_0)$, introduce the auxiliary scheme
$\bigl(\mathcal{Z}_n\bigr)_{n\ge 0}$ defined as follows: for all $n\ge 0$,
\begin{equation}\label{eq:Zscheme}
\mathcal{Z}_{n+1}=\IA_\tau\mathcal{Z}_n+\int_{t_n}^{t_{n+1}}\IB_{t_{n+1}-s}\,\dd\WW(s)
\end{equation}
with initial value $\mathcal{Z}_0=0$, using the same notation as for the general expression~\eqref{eq:LTscheme} of the three splitting schemes~\eqref{eq:LTexact},~\eqref{eq:LTexpo} and~\eqref{eq:LTimp}. One has the following moment bounds for the solution of the scheme~\eqref{eq:Zscheme}. Recall that one has $T=N\tau$ for some integer $N\in\N$.
\begin{lemma}\label{lem:momentboundsZscheme}
For all $T\in(0,\infty)$ and $p\in[1,\infty)$, one has
\begin{equation}\label{eq:momentboundsZscheme}
\underset{\tau\in(0,\tau_0)}\sup~\underset{0\le n\le N}\sup~\E[\|\mathcal{Z}_n\|_\EE^p]<\infty.
\end{equation}
\end{lemma}
Lemma~\ref{lem:momentboundsZscheme} is a variant of~\cite[Lemma~3.5]{MR3986273}, using the same arguments as in the sketch of proof of Lemma~\ref{lem:Zbound} above. The proof of Lemma~\ref{lem:momentboundsZscheme} is therefore omitted.

We are now in position to provide the proof of Theorem~\ref{theo:momentbounds}.
\begin{proof}[Proof of Theorem~\ref{theo:momentbounds}]
For all $n\in\{0,\ldots,N\}$, set
\begin{equation}\label{eq:rn}
r_n=X_n-\mathcal{Z}_n.
\end{equation}
Using the definitions~\eqref{eq:LTscheme} and~\eqref{eq:Zscheme} and the definition~\eqref{eq:psitau} of the mapping $\psi_\tau$, for all $n\in\{0,\ldots,N-1\}$, one has
\begin{align*}
r_{n+1}&=X_{n+1}-\mathcal{Z}_{n+1}=\IA_\tau\Bigl(\phi_\tau(X_n)-\mathcal{Z}_n\Bigr)\\
&=\IA_\tau\Bigl(\phi_\tau(r_n+\mathcal{Z}_n)-\phi_\tau(\mathcal{Z}_n)\Bigr)+\tau\IA_\tau\psi_\tau(\mathcal{Z}_n).
\end{align*}
On the one hand, using the inequalities~\eqref{eq:propSemiGrp-normes} and~\eqref{eq:SemiGrp-normes-imp} and the global Lipschitz continuity property~\eqref{eq:phitau-Lip} of $\phi_\tau$ (see Proposition~\ref{propo:phitau}), one has
\[
\|\IA_\tau\Bigl(\phi_\tau(r_n+\mathcal{Z}_n)-\phi_\tau(\mathcal{Z}_n)\Bigr)\|_\EE \le \|\phi_\tau(r_n+\mathcal{Z}_n)-\phi_\tau(\mathcal{Z}_n)\|_\EE \le e^{\tau(1+\vvvert B\vvvert)}\|r_n\|_\EE.
\]
On the other hand, using the inequalities~\eqref{eq:propSemiGrp-normes} and~\eqref{eq:SemiGrp-normes-imp}, the local Lipschitz continuity property~\eqref{eq:psitau-localLip} of $\psi_\tau$ (see Proposition~\ref{propo:psitau}) and the upper bound~\eqref{eq:psitau0}, one has
\[
\|\IA_\tau\psi_\tau(\mathcal{Z}_n)\|_\EE\le C(\tau_0)\bigl(1+\|\mathcal{Z}_n\|_\EE^4\bigr).
\]
Therefore one obtains the following inequality
\[
\|r_{n+1}\|_\EE\le e^{\tau(1+\vvvert B\vvvert)}\|r_n\|_\EE+C(\tau_0)\bigl(1+\|\mathcal{Z}_n\|_\EE^4\bigr),
\]
and by a straightforward argument, using the fact that $N\tau=T$, one has the estimate:
\[
\|r_n\|_\EE\le C(T,\tau_0)\Bigl(\|r_0\|_\EE+\sum_{k=0}^{n-1}\bigl(1+\|\mathcal{Z}_k\|_\EE^4\bigr)\Bigr),
\]
for all $n\in\{0,\ldots,N\}$.

Finally, for all $p\in[1,\infty)$, using the moment bound~\eqref{eq:momentboundsZscheme} from Lemma~\ref{lem:momentboundsZscheme}, one obtains for all $n\in\{0,\ldots,N\}$
\[
\bigl(\E[\|r_n\|_\EE^p]\bigr)^{\frac1p}\le C(T,\tau_0)\Bigl(\|r_0\|_\EE+\sum_{k=0}^{n-1}
\bigl(1+\bigl(\E[\|\mathcal{Z}_k\|_\EE^{4p}]\bigr)^{\frac1p}\bigr)\Bigr)
\le C_p(T,\tau_0)\Bigl(\|r_0\|_\EE+1\Bigr).
\]
Since $X_n=r_n+\mathcal{Z}_n$ owing to~\eqref{eq:rn}, using the moment bound above and the moment bound~\eqref{eq:momentboundsZscheme} from Lemma~\ref{lem:momentboundsZscheme} then concludes the proof of the moment bound~\eqref{eq:momentboundsLTscheme}. The proof of Theorem~\ref{theo:momentbounds} is thus completed.
\end{proof}

\subsection{Proof of Theorem~\ref{theo:error}}

Recall that the numerical scheme is given by~\eqref{eq:LTscheme}. It is straightforward to check that for all $n\ge 0$ one has
\begin{equation}\label{eq:mildnum}
X_{n}=\IA_\tau^n x_0+\tau\sum_{k=0}^{n-1}\IA_\tau^{n-k}\psi_\tau(X_k)+\sum_{k=0}^{n-1}\int_{t_k}^{t_{k+1}}\IA_\tau^{n-k-1}\IB_{t_{k+1}-s}\,\dd\WW(s).
\end{equation}
Let us introduce the auxiliary process $\bigl(X_n^{\rm aux}\bigr)_{n\ge 0}$ which is defined as follows: for all $n\ge 0$ one has
\begin{equation}\label{eq:mildnumaux}
X_{n}^{\rm aux}=\IA_\tau^n x_0+\tau\sum_{k=0}^{n-1}\IA_\tau^{n-k}\psi_\tau(X_\tau(t_k))+\sum_{k=0}^{n-1}\int_{t_k}^{t_{k+1}}\IA_\tau^{n-k-1}\IB_{t_{k+1}-s}\,\dd\WW(s),
\end{equation}
where we recall that $t_k=k\tau$ and that $\bigl(X_\tau(t)\bigr)_{t\ge 0}$ is the unique mild solution of the auxiliary stochastic evolution equation~\eqref{eq:SEEaux}.
Note that for all $n\ge 0$ one has
\begin{equation}\label{eq:numaux}
X_{n+1}^{\rm aux}=\IA_\tau X_n^{\rm aux}+\tau\IA_\tau \psi_\tau(X_\tau(t_n))+\int_{t_n}^{t_{n+1}}\IB_{t_{n+1}-s}\,\dd\WW(s).
\end{equation}

\begin{lemma}\label{lem:momentboundsnumaux}
For all $T\in(0,\infty)$ and $p\in[1,\infty)$, there exists $C_p(T)\in(0,\infty)$ such that for all $x_0\in\EE$ one has
\begin{equation}\label{eq:momentboundsnumaux}
\underset{\tau\in(0,\tau_0)}\sup~\underset{0\le n\le N}\sup~\E[\|X_n^{\rm aux}\|_\EE^p]\le C_p(T)\bigl(1+\|x_0\|_\EE^p\bigr).
\end{equation}
\end{lemma}

\begin{proof}[Proof of Lemma~\ref{lem:momentboundsnumaux}]
Using the discrete mild formulation~\eqref{eq:mildnumaux} of $X_n^{\rm aux}$, the inequalities~\eqref{eq:propSemiGrp-normes} and~\eqref{eq:SemiGrp-normes-imp}, the local Lipschitz continuity property~\eqref{eq:psitau-localLip} of $\psi_\tau$ and the upper bound~\eqref{eq:psitau0} (see Proposition~\ref{propo:psitau}), for all $\tau\in(0,\tau_0)$ and $n\ge 0$ one has
\[
\|X_n^{\rm aux}\|_\EE\le \|x_0\|_\EE+C(\tau_0)\tau\sum_{k=0}^{n-1}\bigl(1+\|X_\tau(t_k)\|_\EE^4\bigr)+\|\mathcal{Z}_n\|_\EE.
\]
It suffices to use the moment bounds~\eqref{eq:momentSEEaux} for the auxiliary process $X_\tau$ from Proposition~\ref{propo:SEEaux}
and~\eqref{eq:momentboundsZscheme} for the Gaussian random variables $\mathcal{Z}_n$ from Lemma~\ref{lem:momentboundsZscheme},
and the Minkowskii inequality, to conclude the proof of the moment bounds~\eqref{eq:momentboundsnumaux}.
The proof of Lemma~\ref{lem:momentboundsnumaux} is thus completed.
\end{proof}

Observe that for all $n\in\{0,\ldots,N\}$ the error $X(t_n)-X_n$ can be decomposed as follows:
\begin{equation}\label{eq:decomperror}
X(t_n)-X_n=X(t_n)-X_\tau(t_n)+X_\tau(t_n)-X_n^{\rm aux}+X_n^{\rm aux}-X_n.
\end{equation}
In order to prove Theorem~\ref{theo:error}, it suffices to prove error bounds for the three error terms appearing in the right-hand side of~\eqref{eq:decomperror}.
They are given in Lemma~\ref{lem:error1}, Lemma~\ref{lem:error2} and Lemma~\ref{lem:error3} respectively.
The proofs of these technical lemmas are presented at the end of the section.

\begin{lemma}\label{lem:error1}
For all $T\in(0,\infty)$ and $p\in[1,\infty)$, there exists $C_p(T,\tau_0)\in(0,\infty)$ such that for all $x_0\in\EE$ and all $\tau\in(0,\tau_0)$, one has
\begin{equation}\label{eq:error1}
\underset{t\in[0,T]}\sup~\bigl(\E[\|X(t)-X_\tau(t)\|_\HH^p]\bigr)^{\frac{1}{p}}\le C_p(T,\tau_0)\tau\bigl(1+\|x_0\|_\EE^5\bigr).
\end{equation}
\end{lemma}

\begin{lemma}\label{lem:error2}
For all $T\in(0,\infty)$, $p\in[1,\infty)$ and $\alpha\in[0,\frac14)$, there exists $C_{\alpha,p}(T)\in(0,\infty)$ such that for all $x_0=(u_0,v_0)\in\HH^{2\alpha}\cap\EE$, all $\tau\in(0,\tau_0)$, one has
\begin{equation}\label{eq:error2}
\underset{0\le n\le N}\sup~\bigl(\E[\|X_\tau(t_n)-X_n^{\rm aux}\|_\HH^p]\bigr)^{\frac1p}\le C_{\alpha,p}(T)\tau^{\alpha}\bigl(1+\|(-\Delta)^\alpha u_0\|_H^7+\|x_0\|_\EE^7\bigr).
\end{equation}
\end{lemma}

\begin{lemma}\label{lem:error3}
For all $T\in(0,\infty)$, $p\in[1,\infty)$ and $\alpha\in[0,\frac14)$, there exists $C_{\alpha,p}(T)\in(0,\infty)$ such that for all $x_0=(u_0,v_0)\in\HH^{2\alpha}\cap\EE$, all $\tau\in(0,\tau_0)$, one has
\begin{equation}\label{eq:error3}
\underset{0\le n\le N}\sup~\bigl(\E[\|X_n^{\rm aux}-X_n\|_\HH^p]\bigr)^{\frac1p}\le C_{\alpha,p}(T)\tau^{\alpha}\bigl(1+\|(-\Delta)^\alpha u_0\|_H^7+\|x_0\|_\EE^7\bigr).
\end{equation}
\end{lemma}

With the auxiliary error estimates given above, it is straightforward to give the proof of Theorem~\ref{theo:error}.
\begin{proof}[Proof of Theorem~\ref{theo:error}]
Using the decomposition of the error~\eqref{eq:decomperror}, using the Minkowskii inequality and the error estimates~\eqref{eq:error1},~\eqref{eq:error2} and~\eqref{eq:error3}, one obtains the following result: for all $\alpha\in[0,\frac14)$ and $p\in[1,\infty)$, there exists $C_{\alpha,p}\in(0,\infty)$ such that for all $\tau\in(0,\tau_0)$ one has
\begin{align*}
\underset{0\le n\le N}\sup~\bigl(\E[\|X(t_n)-X_n\|_\HH^p]\bigr)^{\frac1p}&\le \underset{0\le n\le N}\sup~\bigl(\E[\|X(t_n)-X_\tau(t_n)\|_\HH^p]\bigr)^{\frac1p}\\
&+\underset{0\le n\le N}\sup~\bigl(\E[\|X_\tau(t_n)-X_n^{\rm aux}\|_\HH^p]\bigr)^{\frac1p}\\
&+\underset{0\le n\le N}\sup~\bigl(\E[\|X_n^{\rm aux}-X_n\|_\HH^p]\bigr)^{\frac1p}\\
&\le C_p(T,\tau_0)\tau\bigl(1+\|x_0\|_\EE^5\bigr)\\
&+C_{\alpha,p}(T)\tau^{\alpha}\bigl(1+\|(-\Delta)^\alpha u_0\|_H^7+\|x_0\|_\EE^7\bigr)\\
&+C_{\alpha,p}(T)\tau^{\alpha}\bigl(1+\|(-\Delta)^\alpha u_0\|_H^7+\|x_0\|_\EE^7\bigr)\\
&\le C_{\alpha,p}(T)\tau^{\alpha}\bigl(1+\|(-\Delta)^\alpha u_0\|_H^7+\|x_0\|_\EE^7\bigr).
\end{align*}
This concludes the proof of the inequality~\eqref{eq:error} and the proof of Theorem~\ref{theo:error} is thus completed.
\end{proof}

Let us now give the proofs of the auxiliary error estimates. Note that the proof of Lemma~\ref{lem:error3} requires the error estimate~\eqref{eq:error2} from Lemma~\ref{lem:error2}.

\begin{proof}[Proof of Lemma~\ref{lem:error1}]
For all $t\ge 0$ and $\tau\in(0,\tau_0)$, set
\[
R_\tau(t)=X_\tau(t)-X(t).
\]
The auxiliary process $\bigl(R_\tau(t)\bigr)_{t\ge 0}$ is the unique solution of the evolution equation
\begin{align*}
\frac{\dd R_\tau(t)}{\dd t}&=-\IL R_\tau(t)+\psi_\tau(X_\tau(t))-\psi_\tau(X(t))+\psi_\tau(X(t))-F(X(t))
\end{align*}
with the initial value $R_\tau(0)=0$. Therefore one obtains, almost surely, for all $t\ge 0$
\begin{align*}
\frac12\frac{\dd \|R_\tau(t)\|_{\HH}^2}{\dd t}&=\langle R_\tau(t),-\Lambda R_\tau(t)\rangle_{\HH}+\langle R_\tau(t),\psi_{\tau}(X_\tau(t))-\psi_{\tau}(X(t))\rangle_{\HH}\\
&\quad +\langle R_\tau(t), \psi_\tau(X(t))-F(X(t))\rangle_{\HH}.
\end{align*}
First, one has
\[
\langle R_\tau(t),-\Lambda R_\tau(t)\rangle_{\HH}\le 0.
\]
Second, using the one-sided Lipschitz continuity property~\eqref{eq:psitau-onesidedLip} from Proposition~\ref{propo:psitau} for $\psi_\tau$ (uniformly with respect to $\tau\in(0,\tau_0)$), one has
\[
\langle R_\tau(t),\psi_{\tau}(X_\tau(t))-\psi_{\tau}(X(t))\rangle_{\HH}\le C(\tau_0)\|R_\tau(t)\|_{\HH}^2.
\]
Finally, using the Cauchy--Schwarz and Young inequalities and the error estimate~\eqref{eq:psitau-cv} from Proposition~\ref{propo:psitau} , one has
\begin{align*}
\langle R_\tau(t), \psi_\tau(X(t))-F(X(t))\rangle_{\HH}&\le \|R_\tau(t)\|_\HH \|\psi_\tau(X(t))-F(X(t))\|_\HH\\
&\le \frac12\|R_\tau(t)\|_\HH^2+\frac12\|\psi_\tau(X(t))-F(X(t))\|_\HH^2\\
&\le \frac12\|R_\tau(t)\|_\HH^2+C(\tau_0)\tau^2\bigl(1+\|X(t)\|_\EE^{10}\bigr).
\end{align*}
Gathering the upper bounds above and using Gronwall's lemma, one obtains, almost surely, for all $t\in[0,T]$
\[
\|R_\tau(t)\|_\HH^2\le C(T,\tau_0)\tau^2\int_{0}^{T}\bigl(1+\|X(s)\|_\EE^{10}\bigr)\,\dd s.
\]
Using the moment bound~\eqref{eq:momentSEE} from Proposition~\ref{propo:SEE}, one then obtains for all $t\in[0,T]$ and all $p\in[2,\infty)$
\begin{align*}
\bigl(\E[\|R_\tau(t)\|_\HH^p]\bigr)^{\frac{2}{p}}&\le C(T,\tau_0)\tau^2\int_{0}^{T}\bigl(1+\E[\|X(s)\|_\EE^{5p}]^{\frac{2}{p}}\bigr)\,\dd s\\
&\le C(T,\tau_0)\tau^2\bigl(1+\underset{s\in[0,T]}\sup~\E[\|X(s)\|_\EE^{5p}]^{\frac{2}{p}}\bigr)\\
&\le C_p(T,\tau_0)\tau^2\bigl(1+\|x_0\|_\EE^{10}\bigr).
\end{align*}
This estimate has been proved for $p\in[2,\infty)$, however it is also valid for $p\in[1,2)$.
This concludes the proof of the error estimate~\eqref{eq:error1} and of Lemma~\ref{lem:error1}.
\end{proof}

In order to prove Lemma~\ref{lem:error2}, let us recall the following useful standard inequality:
\begin{equation}\label{eq:ineq}
\underset{n\in\N,z\in[0,\infty)}\sup~n|\frac{1}{(1+z)^n}-e^{-nz}|+\underset{n\in\N,z\in[0,\infty)}\sup~\frac{|\frac{1}{(1+z)^n}-e^{-nz}|}{\min(1,z)}<\infty.
\end{equation}
In addition, for all $\alpha\in[0,1]$, $n\in\N$ and $z\in[0,\infty)$, one has $\min(1,z)\le z^\alpha$.
See Section~\ref{appendixx} in the appendix for a proof.

\begin{proof}[Proof of Lemma~\ref{lem:error2}]
Using the mild formulations~\eqref{eq:SEEauxmild} for $X_\tau(t_n)$ and~\eqref{eq:mildnumaux} for $X_n^{\rm aux}$, one obtains the following decomposition of the error:
for all $n\ge 0$, one has
\begin{equation}\label{eq:decomperror2}
X_\tau(t_n)-X_n^{\rm aux}=E_{n}^{\tau,1}+E_{n}^{\tau,2}+E_{n}^{\tau,3}+E_{n}^{\tau,4}+E_{n}^{\tau,5},
\end{equation}
where
\begin{align}
E_{n}^{\tau,1}&=(e^{-n\tau \IL}-\IA_\tau^{n})x_0\\
E_{n}^{\tau,2}&=\mathcal{Z}(t_n)-\mathcal{Z}_n\\
E_{n}^{\tau,3}&=\sum_{k=0}^{n-1}\int_{t_k}^{t_{k+1}}e^{-(t_n-s)\IL}\bigl(\psi_\tau(X_\tau(s))-\psi_\tau(X_{\tau}(t_{k}))\bigr)\,\dd s\\
E_{n}^{\tau,4}&=\sum_{k=0}^{n-1}\int_{t_k}^{t_{k+1}}\bigl(e^{-(t_n-s)\IL}-e^{-(t_n-t_{k})\IL}\bigr)\psi_\tau(X_{\tau}(t_{k}))\,\dd s\\
E_{n}^{\tau,5}&=\tau\sum_{k=0}^{n-1}\bigl(e^{-(t_n-t_{k})\IL}-\IA_\tau^{n-k}\bigr)\psi_\tau(X_\tau(t_{k})).
\end{align}
Let us now give estimates for those five error terms.

$\bullet$ If the splitting schemes~\eqref{eq:LTexact} and~\eqref{eq:LTexpo} are considered, one has $\IA_\tau=e^{-\tau\IL}$ and thus $E_n^{\tau,1}=0$ for all $n\ge 0$.
If the splitting scheme~\eqref{eq:LTimp} is considered, one has $\IA_\tau=(I+\tau\IL)^{-1}$, thus using the inequality~\eqref{eq:ineq}, for all $n\in\{0,\ldots,N\}$, one has
\begin{align*}
\|E_n^{\tau,1}\|_\HH^2&=\|\bigl(e^{n\tau\Delta}-((I-\tau\Delta)^{-1})^n\bigr)u_0\|_{H}^2\\
&=\sum_{j=1}^{\infty}(\frac{1}{(1+\tau\lambda_j)^n}-e^{-n\tau\lambda_j}\bigr)^2\langle u_0,e_j\rangle_{H}^2\\
&\le C_\alpha\sum_{j=1}^{\infty}(\tau\lambda_j)^{2\alpha}\langle u_0,e_j\rangle_{H}^2\\
&\le C_\alpha \tau^{2\alpha}\|(-\Delta)^\alpha u_0\|_H^2.
\end{align*}
Therefore one obtains the following upper bound: for all $\alpha\in[0,\frac14)$, there exists $C_\alpha\in(0,\infty)$ such that for all $\tau\in(0,\tau_0)$ one has
\begin{equation}\label{eq:En1}
\underset{0\le n\le N}\sup~\bigl(\E[\|E_n^{\tau,1}\|_\HH^p]\bigr)^{\frac1p}\le C_{\alpha}\tau^\alpha\|(-\Delta)^\alpha u_0\|_H
\end{equation}
$\bullet$ Note that if the splitting scheme~\eqref{eq:LTexact} is considered ($X_n=X_n^{\LTexact}$ for all $n\ge 0$), one has $E_n^{\tau,2}=0$ for all $n\ge 0$. If the splitting schemes~\eqref{eq:LTexpo} and~\eqref{eq:LTimp} are considered, for all $n\ge 0$ one has
\[
E_n^{\tau,2}=\mathcal{Z}(t_n)-\mathcal{Z}_n=\begin{pmatrix} Z(t_n)-Z_n\\ 0\end{pmatrix},
\]
with $Z_n=Z_n^{\LTexpo}$ (resp. $Z_n=Z_n^{\LTimp}$) if the scheme~\eqref{eq:LTexpo} (resp. the scheme~\eqref{eq:LTimp}) is considered. Here, we denote
\begin{align*}
Z_{n+1}^{\LTexpo}&=e^{\tau\Delta}\Bigl(Z_{n}^{\LTexpo}+\delta W_n\Bigr)\\
Z_{n+1}^{\LTimp}&=(I-\tau\Delta)^{-1}\Bigl(Z_{n}^{\LTimp}+\delta W_n\Bigr).
\end{align*}
One has the following mean-square error estimate, which are standard results in the analysis of numerical schemes
for parabolic semilinear stochastic partial differential equations, see for instance \cite[Theorem~3.2]{MR1873517}:
for all $\alpha\in[0,\frac14)$, there exists $C_\alpha\in(0,\infty)$ such that
\[
\underset{n\ge 0}\sup~\E[\|Z(t_n)-Z_n\|_H^2]\le C_\alpha \tau^{2\alpha},
\]
if $Z_n=Z_n^{\LTexpo}$ and $Z_n=Z_n^{\LTimp}$. Since $Z(t_n)-Z_n$ is a $H$-valued Gaussian random variable, one obtains the following upper bound: for all $\alpha\in[0,\frac14)$ and $p\in[1,\infty)$, there exists $C_{\alpha,p}\in(0,\infty)$ such that for all $\tau\in(0,\tau_0)$ one has
\begin{equation}\label{eq:En2}
\underset{0\le n\le N}\sup~\bigl(\E[\|E_n^{\tau,2}\|_\HH^p]\bigr)^{\frac1p}\le C_{\alpha,p}\tau^\alpha.
\end{equation}
$\bullet$ Using the inequality~\eqref{eq:propSemiGrp-normes} and the local Lipschitz continuity property~\eqref{eq:psitau-localLip} of $\psi_\tau$ (Proposition~\ref{propo:psitau}),
one obtains
\begin{align*}
\|E_n^{\tau,3}\|_\HH&\le \sum_{k=0}^{n-1}\int_{t_k}^{t_{k+1}}\|e^{-(t_n-s)\IL}\bigl(\psi_\tau(X_\tau(s))-\psi_\tau(X_{\tau}(t_{k}))\bigr)\|_{\HH}\,\dd s\\
&\le \sum_{k=0}^{n-1}\int_{t_k}^{t_{k+1}}\|\bigl(\psi_\tau(X_\tau(s))-\psi_\tau(X_{\tau}(t_{k}))\bigr)\|_{\HH}\,\dd s\\
&\le C(\tau_0)\sum_{k=0}^{n-1}\int_{t_k}^{t_{k+1}}\bigl(1+\|X_\tau(s)\|_\EE^3+\|X_\tau(t_k)\|_\EE^3\bigr)\|X_\tau(s)-X_{\tau}(t_{k})\|_{\HH}\,\dd s.
\end{align*}
Using the Minkowskii and Cauchy--Schwarz inequalities, the moment bound~\eqref{eq:momentSEEaux} (Proposition~\ref{propo:SEEaux}) and the regularity estimate~\eqref{eq:regularityXtau} (Lemma~\ref{lem:regularityXtau}), one has
\begin{align*}
\bigl(\E[\|E_n^{\tau,3}\|_\HH^p]\bigr)^{\frac1p}&\le C(\tau_0)\sum_{k=0}^{n-1}\int_{t_k}^{t_{k+1}}\bigl(1+\underset{r\in[t_{k},t_{k+1}]}\sup~\bigl(\E[\|X_\tau(r)\|_\EE^{6p}]\bigr)^{\frac{1}{2p}}\bigr)\bigl(\E[\|X_\tau(s)-X_{\tau}(t_{k})\|_{\HH}^{2p}]\bigr)^{\frac{1}{2p}}\,\dd s\\
&\le C_{\alpha,p}(T)\tau^\alpha (1+\|x_0\|_\EE^3)\bigl(1+\|(-\Delta)^\alpha u_0\|_H^4+\|x_0\|_\EE^4\bigr).
\end{align*}
Therefore one obtains the following upper bound:
for all $\alpha\in[0,\frac14)$, $p\in[1,\infty)$ and $T\in(0,\infty)$, there exists $C_{\alpha,p}(T)\in(0,\infty)$ such that for all $\tau\in(0,\tau_0)$ one has
\begin{equation}\label{eq:En3}
\underset{0\le n\le N}\sup~\bigl(\E[\|E_n^{\tau,3}\|_\HH^p]\bigr)^{\frac1p}\le C_{\alpha,p}(T)\tau^\alpha\bigl(1+\|(-\Delta)^\alpha u_0\|_H^7+\|x_0\|_\EE^7\bigr).
\end{equation}
$\bullet$ Using the inequality~\eqref{eq:propSemiGrp-regul} from Proposition~\ref{propSemiGrp} (with $\mu=\alpha\in[0,1)$ and $\nu=0$) and the local Lipschitz continuity property~\eqref{eq:psitau-localLip} of $\psi_\tau$ combined with the bound~\eqref{eq:psitau0} (Proposition~\ref{propo:psitau}), one has for all $s\in[t_{k},t_{k+1}]$
\begin{align*}
\|\bigl(e^{-(t_n-s)\IL}-e^{-(t_n-t_{k})\IL}\bigr)\psi_\tau(X_{\tau}(t_{k}))\|_\HH&\le C_{\alpha}\frac{|s-t_k|^\alpha}{(t_n-s)^\alpha}\|\psi_\tau(X_{\tau}(t_{k}))\|_\HH\\
&\le C_{\alpha}\frac{\tau^\alpha}{(t_n-s)^\alpha}\bigl(1+\|X_\tau(t_k)\|_\EE^4\bigr).
\end{align*}
Using the Minkoswskii inequality, the moment bounds~\eqref{eq:momentSEEaux} from Proposition~\ref{propo:SEEaux}, and the fact that $\int_0^T s^{-\alpha}\,\dd s<\infty$ for $\alpha\in[0,1)$, one obtains the following upper bound: for all $\alpha\in[0,\frac14)$, $p\in[1,\infty)$ and $T\in(0,\infty)$, there exists $C_{\alpha,p}(T)\in(0,\infty)$ such that for all $\tau\in(0,\tau_0)$ one has
\begin{equation}\label{eq:En4}
\underset{0\le n\le N}\sup~\bigl(\E[\|E_n^{\tau,4}\|_\HH^p]\bigr)^{\frac1p}\le C_{\alpha,p}(T)\tau^\alpha\bigl(1+\|x_0\|_\EE^4\bigr).
\end{equation}
$\bullet$ Note that if the splitting schemes~\eqref{eq:LTexact} and~\eqref{eq:LTexpo} are considered, one has $\IA_\tau=e^{-\tau\IL}$ and thus $E_n^{\tau,5}=0$ for all $n\ge 0$.
If the splitting scheme~\eqref{eq:LTimp} is considered, one has $\IA_\tau=(I+\tau\IL)^{-1}$. Using the inequality~\eqref{eq:ineq}, for all $x=(u,v)\in\HH$ and all $0\le k\le n-1$ one has
\[
\|(e^{-(t_n-t_k)\IL}x-\IA_\tau^{n-k}x\|_\HH=\|e^{(n-k)\tau\Delta}u-((I-\tau\Delta)^{-1})^{n-k}u\|_H\le \frac{C\|u\|_H}{(n-k)}\le \frac{C\|x\|_\HH}{(n-k)^\alpha}. 
\]
As a consequence, using the Minkowskii inequality, the local Lipschitz continuity property~\eqref{eq:psitau-localLip} of $\psi_\tau$ combined with the bound~\eqref{eq:psitau0} (Proposition~\ref{propo:psitau}) and the moment bounds~\eqref{eq:momentSEEaux} from Proposition~\ref{propo:SEEaux}, one has
\begin{align*}
\bigl(\E[\|E_n^{\tau,5}\|_\HH^p]\bigr)^{\frac1p}&\le \tau\sum_{k=0}^{n-1}\frac{C}{(n-k)^\alpha}\bigl(1+\bigl(\E[\|X_\tau(t_k)\|_\EE^{4p}]\bigr)^{\frac1p}\bigr)\\
&\le C_p(T)\tau\sum_{\ell=1}^{n}\frac{1}{t_{\ell}^\alpha} \tau^\alpha\bigl(1+\|x_0\|_\EE^4\bigr).
\end{align*} 
Using the fact that for all $\alpha\in[0,1)$ one has
\[
\underset{\tau\in(0,\tau_0)}\sup~\tau\sum_{\ell=1}^{N}\frac{1}{t_{\ell}^\alpha}<\infty,
\]
one obtains the following upper bound: for all $\alpha\in[0,\frac14)$, $p\in[1,\infty)$ and $T\in(0,\infty)$, there exists $C_{\alpha,p}(T)\in(0,\infty)$ such that for all $\tau\in(0,\tau_0)$ one has
\begin{equation}\label{eq:En5}
\underset{0\le n\le N}\sup~\bigl(\E[\|E_n^{\tau,5}\|_\HH^p]\bigr)^{\frac1p}\le C_{\alpha,p}(T)\tau^\alpha\bigl(1+\|x_0\|_\EE^4\bigr).
\end{equation}
We are now in position to conclude the proof: using the decomposition of the error~\eqref{eq:decomperror2} and the upper bounds~\eqref{eq:En1},~\eqref{eq:En2},~\eqref{eq:En3},
~\eqref{eq:En4} and~\eqref{eq:En5}, one obtains the following upper bound: for all $\alpha\in[0,\frac14)$, $p\in[1,\infty)$ and $T\in(0,\infty)$, there exists $C_{\alpha,p}(T)\in(0,\infty)$ such that for all $\tau\in(0,\tau_0)$ one has
\[
\underset{0\le n\le N}\sup~\bigl(\E[\|X_\tau(t_n)-X_n^{\rm aux}\|_\HH^p]\bigr)^{\frac1p}\le C_{\alpha,p}(T)\tau^\alpha\bigl(1+\|(-\Delta)^\alpha u_0\|_H^7+\|x_0\|_\EE^7\bigr).
\]
This concludes the proof of the inequality~\eqref{eq:error2} and the proof of Lemma~\ref{lem:error2} is completed.
\end{proof}
Note that the proof of Lemma~\ref{lem:error2} above does not use Gronwall inequalities arguments.

\begin{proof}[Proof of Lemma~\ref{lem:error3}]
Using the expressions~\eqref{eq:numaux} and~\eqref{eq:LTscheme} for $X_n^{\rm aux}$ and $X_n$,
and the definition~\eqref{eq:psitau} of the mapping $\psi_\tau$, for all $n\in\{0,\ldots,N-1\}$ one obtains
\[
X_{n+1}^{\rm aux}-X_{n+1}=\IA_\tau\bigl(X_n^{\rm aux}-X_n\bigr)+\tau \IA_\tau\bigl(\psi_\tau(X_\tau(t_n))-\psi_\tau(X_n)\bigr).
\]
Writing
\[
\psi_\tau(X_\tau(t_n))=\psi_\tau(X_\tau(t_n))-\psi_\tau(X_n^{\rm aux})+\psi_\tau(X_n^{\rm aux}),
\]
and using again the identity~\eqref{eq:psitau}, one obtains
\begin{equation}\label{eq:decomperror3}
X_{n+1}^{\rm aux}-X_{n+1}=\IA_\tau\bigl(\phi_\tau(X_n^{\rm aux})-\phi_\tau(X_n)\bigr)+\tau \IA_\tau\bigl(\psi_\tau(X_\tau(t_n))-\psi_\tau(X_n^{\rm aux})\bigr).
\end{equation}
On the one hand, using the inequalities~\eqref{eq:propSemiGrp-normes} (Proposition~\ref{propSemiGrp}), if $\IA_\tau=e^{-\tau\IL}$ and~\eqref{eq:SemiGrp-normes-imp}, if $\IA_\tau=(I+\tau\IL)^{-1}$, and the global Lipschitz continuity property~\eqref{eq:phitau-Lip} of $\phi_\tau$ (Proposition~\ref{propo:phitau}), one obtains
\begin{align*}
\|\IA_\tau\bigl(\phi_\tau(X_n^{\rm aux})-\phi_\tau(X_n)\bigr)\|_\HH&\le \|\phi_\tau(X_n^{\rm aux})-\phi_\tau(X_n)\|_\HH\\
&\le e^{\tau(1+\vvvert B\vvvert)}\|X_n^{\rm aux}-X_n\|_\HH.
\end{align*}
On the other hand, using the inequalities~\eqref{eq:propSemiGrp-normes} (Proposition~\ref{propSemiGrp}), if $\IA_\tau=e^{-\tau\IL}$ and~\eqref{eq:SemiGrp-normes-imp}, if $\IA_\tau=(I+\tau\IL)^{-1}$, and the local Lipschitz continuity property~\eqref{eq:psitau-localLip} of $\psi_\tau$ (Proposition~\ref{propo:psitau}), one obtains
\begin{align*}
\|\IA_\tau\bigl(\psi_\tau(X_\tau(t_n))-\psi_\tau(X_n^{\rm aux})\bigr)\|_\HH&\le \|\psi_\tau(X_\tau(t_n))-\psi_\tau(X_n^{\rm aux})\|_\HH\\
&\le C(\tau_0)\Bigl(1+\|X_\tau(t_n)\|_\EE^3+\|X_n^{\rm aux}\|_\EE^3\Bigr)\|X_\tau(t_n)-X_n^{\rm aux}\|_\HH.
\end{align*}
By a straightforward argument, since $X_0^{\rm aux}=X_0=x_0$, for all $n\in\{0,\ldots,N\}$, one has
\[
\|X_n^{\rm aux}-X_n\|_\HH\le C(\tau_0)e^{T(1+\vvvert B\vvvert)}\tau\sum_{k=1}^{N}\Bigl(1+\|X_\tau(t_k)\|_\EE^3+\|X_k^{\rm aux}\|_\EE^3\Bigr)\|X_\tau(t_k)-X_k^{\rm aux}\|_\HH.
\]
Using the Minkowskii and Cauchy--Schwarz inequalities, the moment bounds~\eqref{eq:momentSEEaux} and~\eqref{eq:momentboundsnumaux} from Proposition~\ref{propo:SEEaux} and Lemma~\ref{lem:momentboundsnumaux} respectively, and the error estimate~\eqref{eq:error2} from Lemma~\ref{lem:error2}, one obtains the following strong error estimate: for all $\alpha\in[0,\frac14)$, $p\in[1,\infty)$ and $T\in(0,\infty)$, there exists $C_{\alpha,p}(T)\in(0,\infty)$ such that for all $\tau\in(0,\tau_0)$ one has
\begin{align*}
\underset{0\le n\le N}\sup~\bigl(\E[\|X_n^{\rm aux}-X_n\|_\HH^p]\bigr)^{\frac1p}&\le C(T)\tau\sum_{k=1}^{N}\Bigl(1+\bigl(\E[\|X_\tau(t_k)\|_\EE^{6p}])^{\frac{1}{2p}}+\bigl(\|X_k^{\rm aux}\|_\EE^{6p}]\bigr)^{\frac{1}{2p}}\Bigr)\\
&\quad\times \bigl(\E[\|X_\tau(t_k)-X_k^{\rm aux}\|_\HH^{2p}\bigr)^{\frac{1}{2p}}\\
&\le C_{\alpha,p}(T)\tau^{\alpha}\bigl(1+\|x_0\|_\EE^3\bigr)\bigl(1+\|(-\Delta)^\alpha u_0\|_H^4+\|x_0\|_\EE^4\bigr).
\end{align*}
This concludes the proof of the inequality~\eqref{eq:error3} and the proof of Lemma~\ref{lem:error3} is thus completed.
\end{proof}

\section{Numerical experiments}\label{sect-num}
This section presents numerical experiments to support and illustrate
the above theoretical results. To perform these numerical experiments,
we consider the stochastic FitzHugh--Nagumo SPDE system~\eqref{eq:FhNsystem} with Neumann boundary conditions on the interval $[0,1]$.
The spatial discretization is performed using a standard finite difference method with mesh size denoted by $h$.
In order to obtain a linear system with a symmetric matrix, we use centered differences
for the numerical discretization of the Laplacian, while first order
differences are used for the discretization of the Neumann boundary conditions.
The initial values are given by $u_0(\zeta)=\cos(2\pi \zeta)$ and $v_0(\zeta)=\cos(2\pi \zeta)$.
For the temporal discretization, we use the three Lie--Trotter splitting integrators~\eqref{eq:LTexact},~\eqref{eq:LTexpo} and~\eqref{eq:LTimp} studied in this paper,
denoted below by \textsc{LTexact, LTexpo, LTimp} respectively.

\subsection{Evolution plots}
Let us first display one sample of the numerical solutions of the stochastic FitzHugh--Nagumo system~\eqref{eq:FhNsystem}
with the parameters $\gamma_1=0.08$, $\gamma_2=0.8\gamma_1$ and $\beta=0.7$. The SPDE is discretized with
finite differences with mesh $h=2^{-10}$. We consider the time interval $[0,T]=[0,1]$ and apply the integrators
with time step size $\tau=2^{-15}$. The results are presented in Figure~\ref{fig:evol}. The general behaviour of the numerical solutions given by the three splitting schemes is the same.
However, one can observe a spatial smoothing effect in the $u$ component of the solution when the schemes \textsc{LTexpo}--\eqref{eq:LTexact} or to some extent \textsc{LTimp}--\eqref{eq:LTimp} are applied: for a given time step size, the spatial regularity of the numerical solution is increased compared with the one of the exact solution. On the contrary, the scheme \textsc{LTexact}--\eqref{eq:LTexact} preserves the spatial regularity of the solution for any value of the time step size. We refer to the recent preprint~\cite{CEB-modified} for the analysis of this phenomenon for parabolic semilinear SPDEs. Let us emphasize that the phenomenon is due to the way the stochastic convolution is computed, exactly for the scheme \textsc{LTexact}--\eqref{eq:LTexact} or approximately for the schemes \textsc{LTexpo}--\eqref{eq:LTexact} and \textsc{LTimp}--\eqref{eq:LTimp}.

\begin{figure}
\centering
 \begin{subfigure}[b]{0.35\textwidth}
   \includegraphics[width=\textwidth]{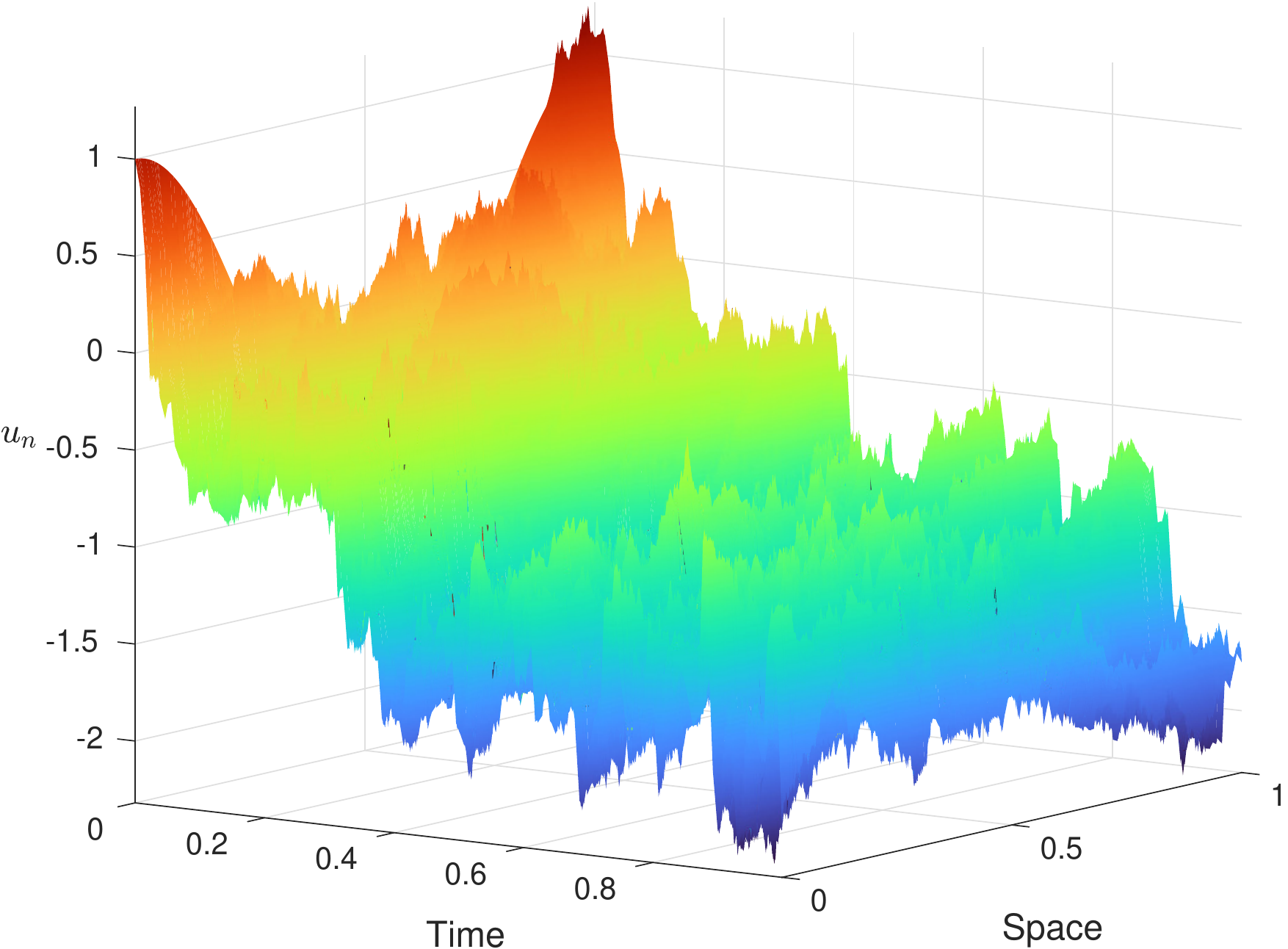}
    \caption{Component $u$ for \textsc{LTexact}}
 \end{subfigure}
 \begin{subfigure}[b]{0.35\textwidth}
   \includegraphics[width=\textwidth]{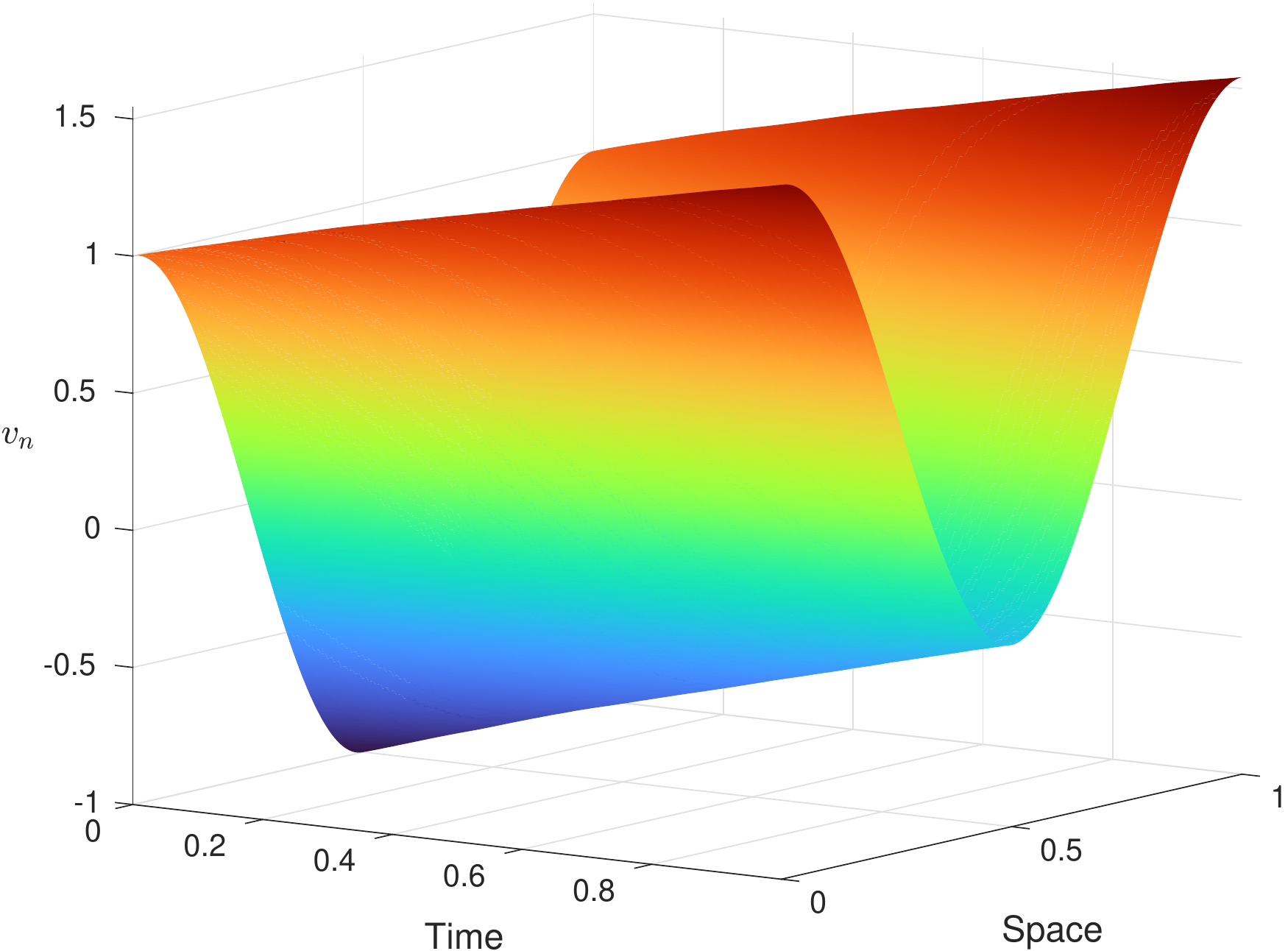}
   \caption{Component $v$ for \textsc{LTexact}}
 \end{subfigure}
     ~

 \begin{subfigure}[b]{0.35\textwidth}
   \includegraphics[width=\textwidth]{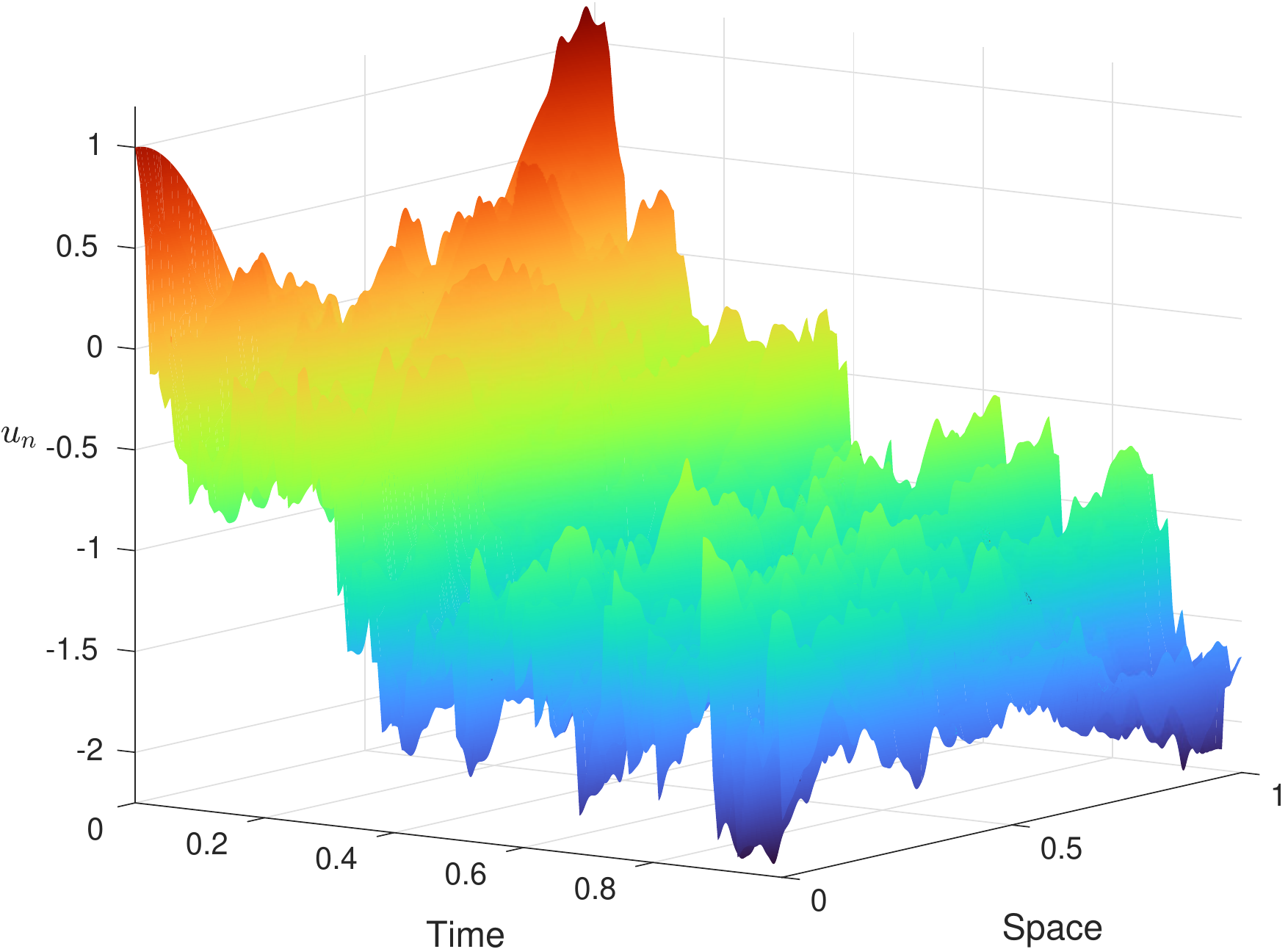}
   \caption{Component $u$ for \textsc{LTexpo}}
 \end{subfigure}
 \begin{subfigure}[b]{0.35\textwidth}
   \includegraphics[width=\textwidth]{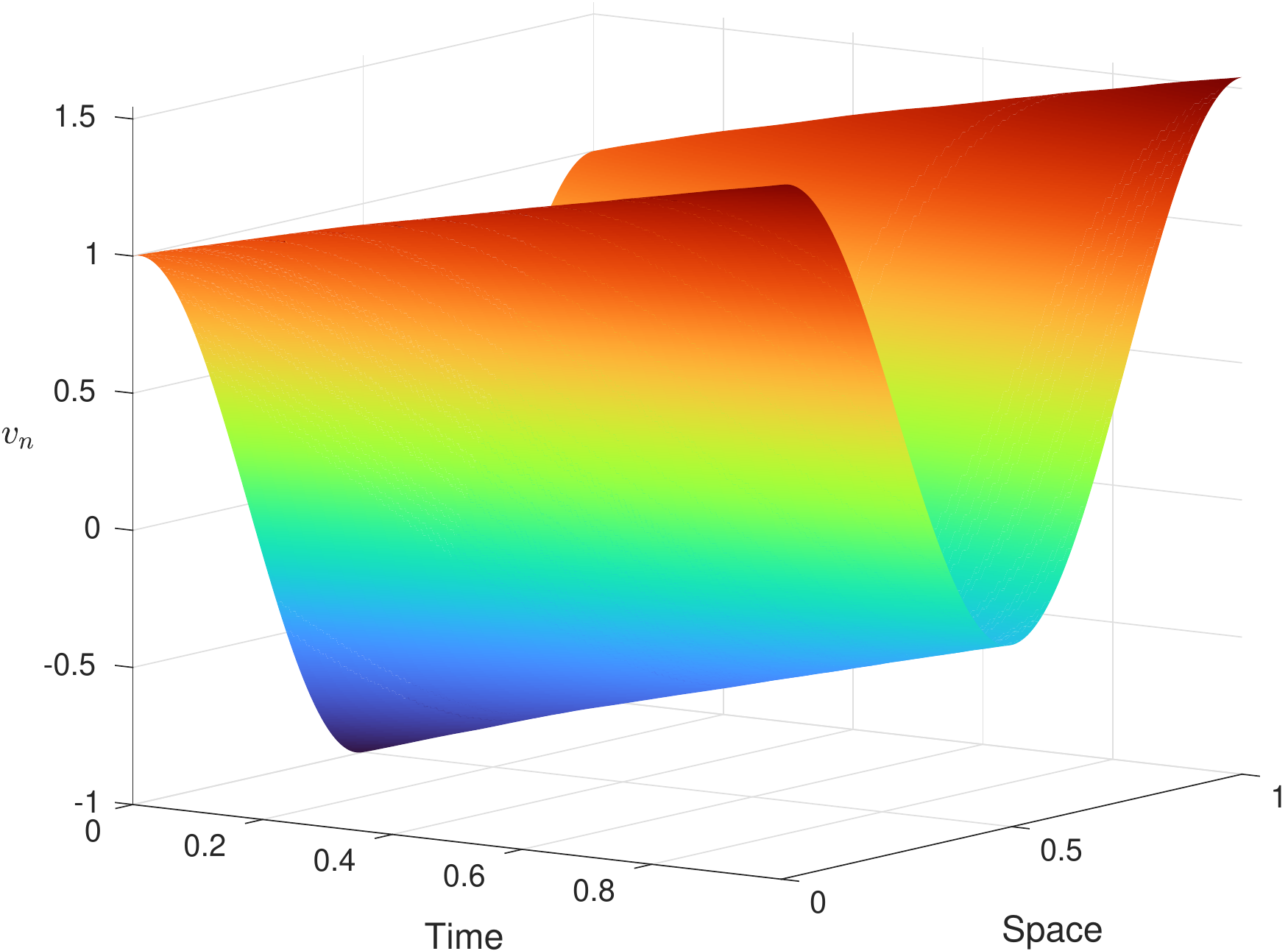}
   \caption{Component $v$ for \textsc{LTexpo}}
 \end{subfigure}
       ~

 \begin{subfigure}[b]{0.35\textwidth}
   \includegraphics[width=\textwidth]{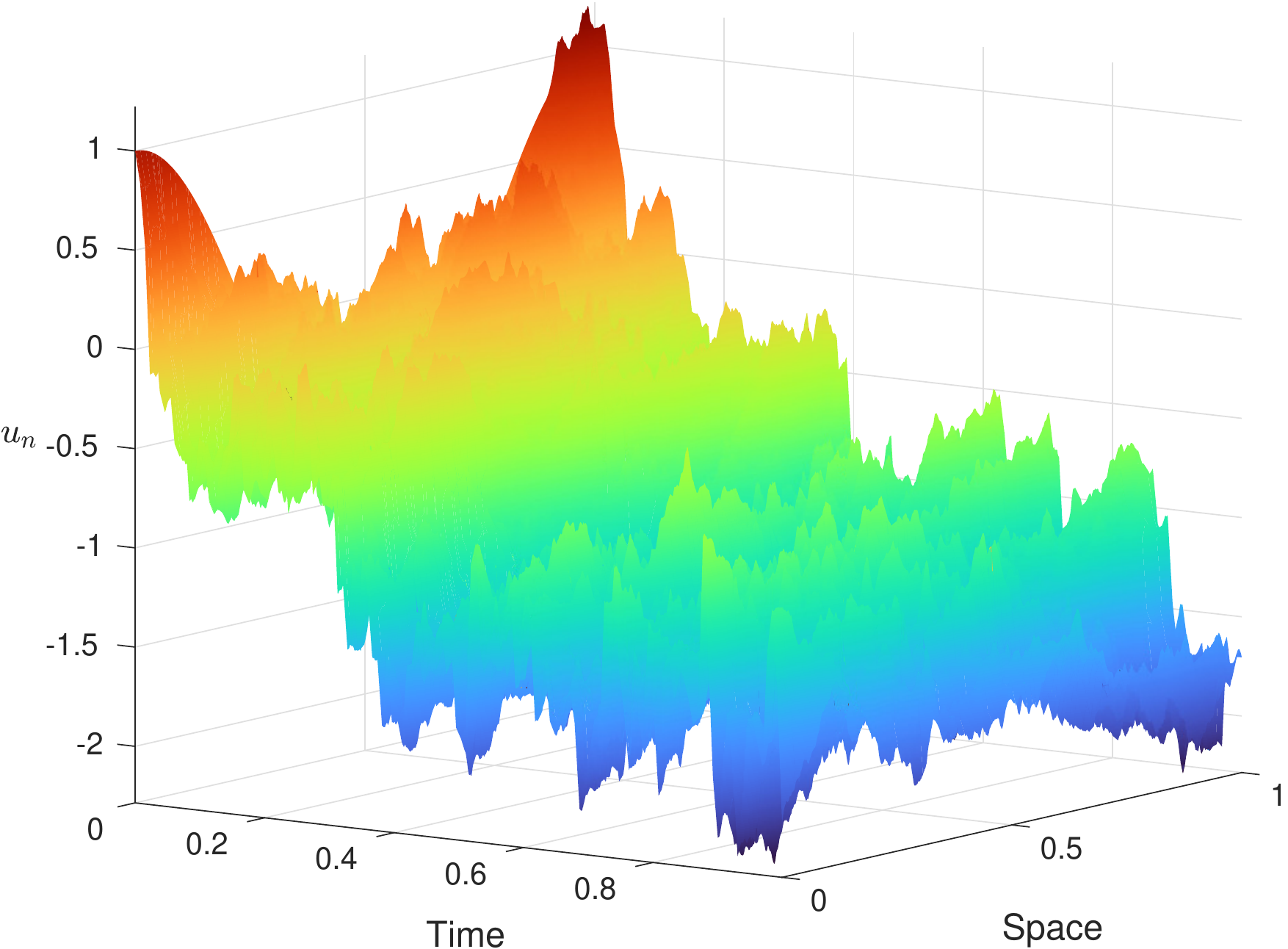}
   \caption{Component $u$ for \textsc{LTimp}}
 \end{subfigure}
 \begin{subfigure}[b]{0.35\textwidth}
   \includegraphics[width=\textwidth]{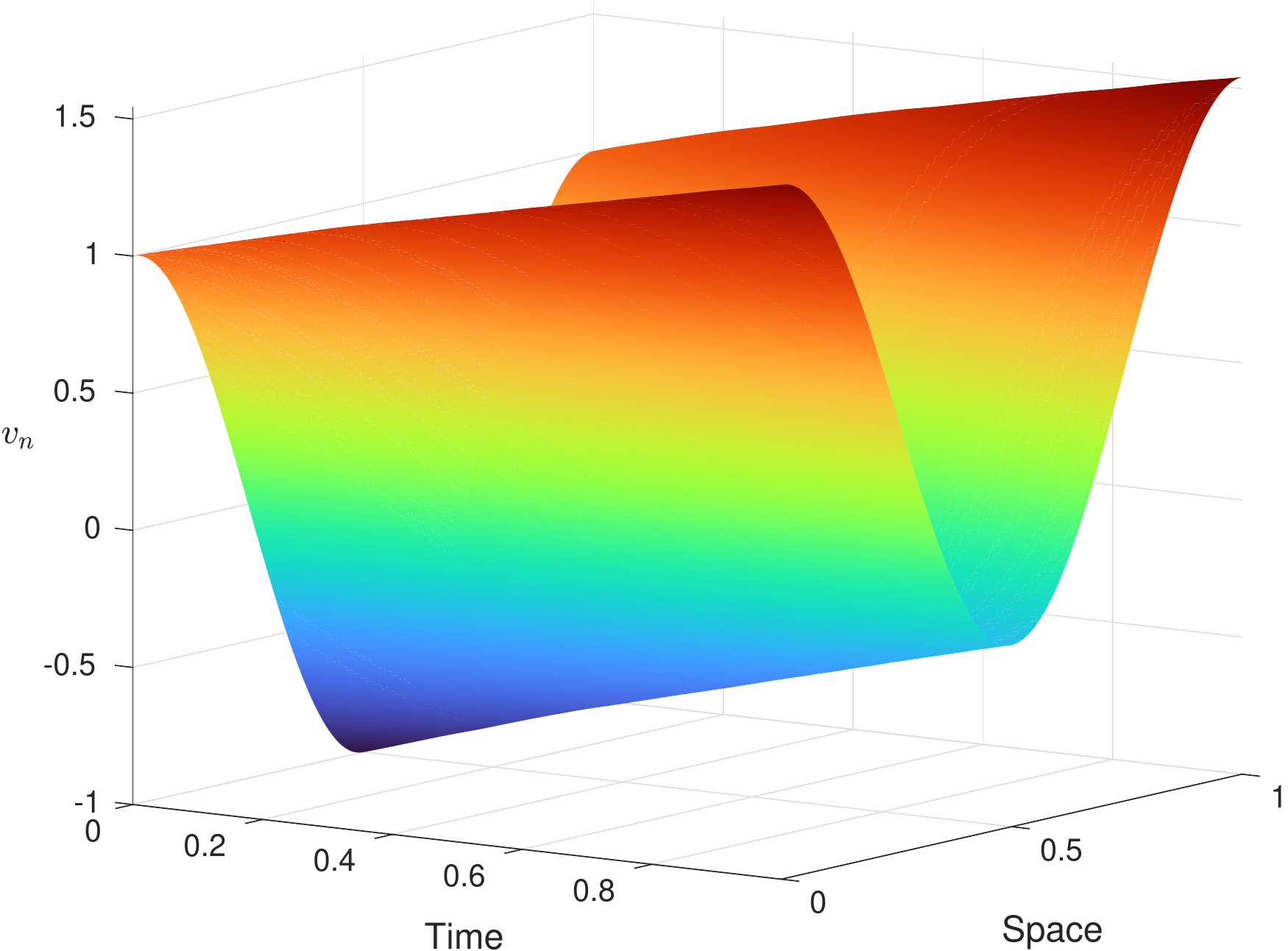}
   \caption{Component $v$ for \textsc{LTimp}}
 \end{subfigure}

\caption{Space-time evolution plots of $u$ and $v$ using the Lie--Trotter splitting schemes \textsc{LTexact}, \textsc{LTexpo}, and \textsc{LTimp}.}
\label{fig:evol}
\end{figure}

\subsection{Mean-square error plots}
To illustrate the rates of strong convergence for the Lie--Trotter splitting schemes stated in Theorem~\ref{theo:error},
we consider the stochastic FitzHugh--Nagumo system~\eqref{eq:FhNsystem} with the parameters $\gamma_1=\gamma_2=\beta=1$, with $T=1$ and apply a finite difference method with $h=2^{-9}$ for spatial discretization. We apply the Lie--Trotter splitting schemes with time steps
ranging from $2^{-10}$ to $2^{-18}$. The reference solution is computed using the scheme~\textsc{LTexact}--\eqref{eq:LTexact} with time step size $\tau_{\rm ref}=2^{-18}$. The expectation is approximated
using $M_s=100$ samples. We have checked that the Monte Carlo error is negligible.
A plot in logarithmic scales for the mean-square errors
\[
\bigl(\E[\|X(t_N)-X_N\|_\HH^2]\bigr)^{\frac12}
\]
is given on the left-hand side of Figure~\ref{fig:params}. We observe that the strong rate of convergence for the three considered Lie--Trotter splitting schemes is at least $1/4$, which illustrates the result stated in Theorem~\ref{theo:error}. Furthermore, the numerical experiments suggest that for the scheme~\textsc{LTexact}--\eqref{eq:LTexact} the order of convergence is $1/2$, which is not covered by Theorem~\ref{theo:error}. The fact that using an accelerated exponential Euler scheme where the stochastic convolution is computed exactly yields higher order of convergence is known for parabolic semilinear stochastic PDEs driven by space-time white noise, under appropriate conditions, see for instance~\cite{jk09} or~\cite[Proposition~7.3]{CEB-modified}. However, the stochastic FitzHugh--Nagumo equations considered in this article are not parabolic systems therefore it is not known how to prove the observed higher order strong rate of convergence. This question may be studied in future works.

The right-hand side of Figure~\ref{fig:params} shows the errors for the variant~\eqref{eq:LTschemehat} of the splitting scheme~\eqref{eq:LTscheme} introduced in Remark~\ref{rem:schemehat}: the mapping $\phi_\tau=\phi_\tau^{\rm L}\circ\phi_\tau^{\rm NL}$ given by~\eqref{eq:phitau} is replaced by $\hat{\phi}_\tau=\phi_\tau^{\rm NL}\circ\phi_\tau^{\rm L}$ given by~\eqref{eq:phitauhat}. As explained in Remark~\ref{rem:schemehat}, this type of Lie--Trotter schemes is not covered by the results in Section~\ref{sec-main}, more precisely the moment bounds in Theorem~\ref{theo:momentbounds} cannot be proved by the techniques used in this article. However, the numerical experiments are similar to those on the left-hand side of Figure~\ref{fig:params} and suggest that the strong order of convergence for this variant is at least $1/4$, and that higher order convergence with rate $1/2$ may be obtained for the variant of the scheme~\textsc{LTexact}--\eqref{eq:LTexact}.

\begin{figure}
\centering
\includegraphics[height=5cm,keepaspectratio]{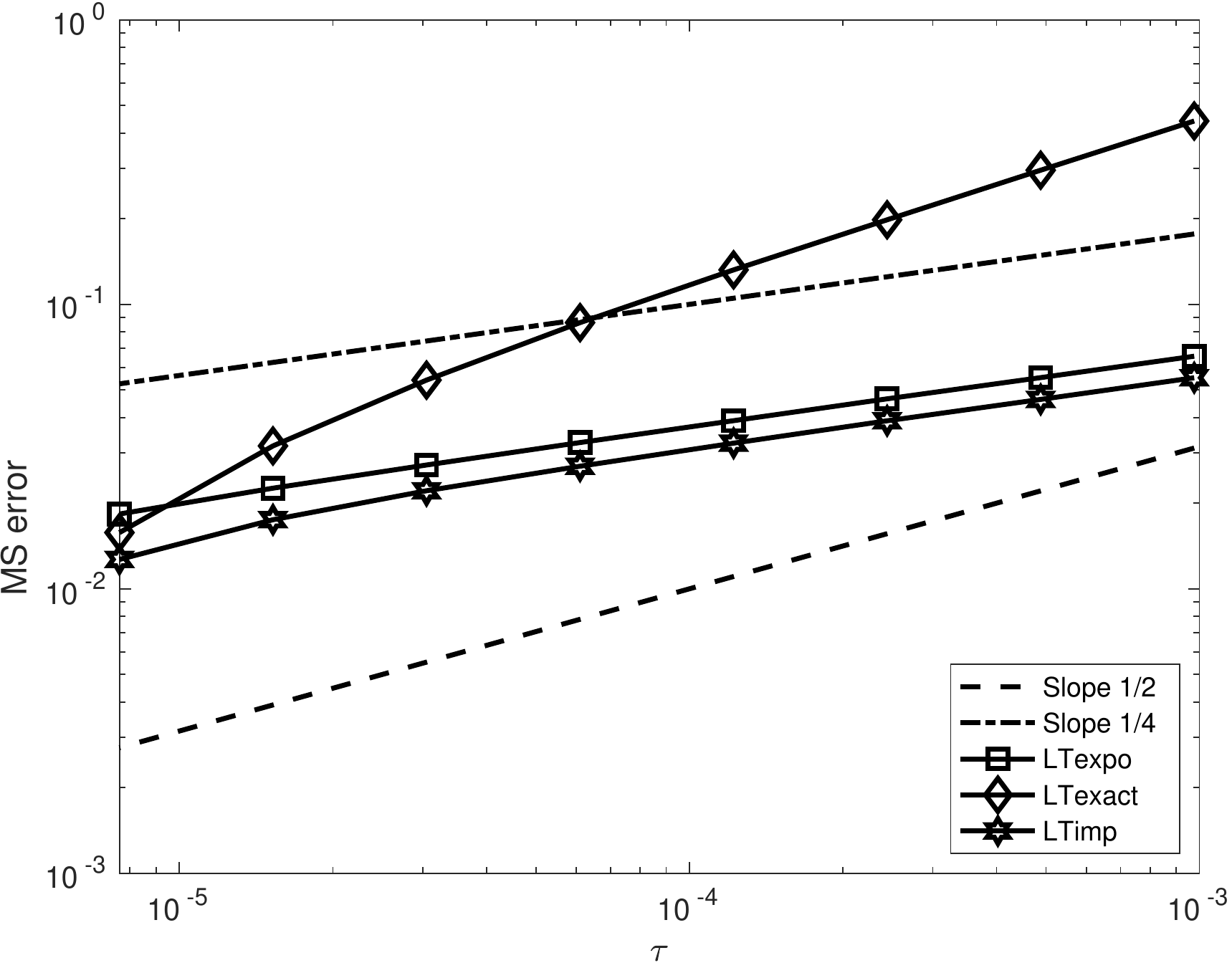}
\includegraphics[height=5cm,keepaspectratio]{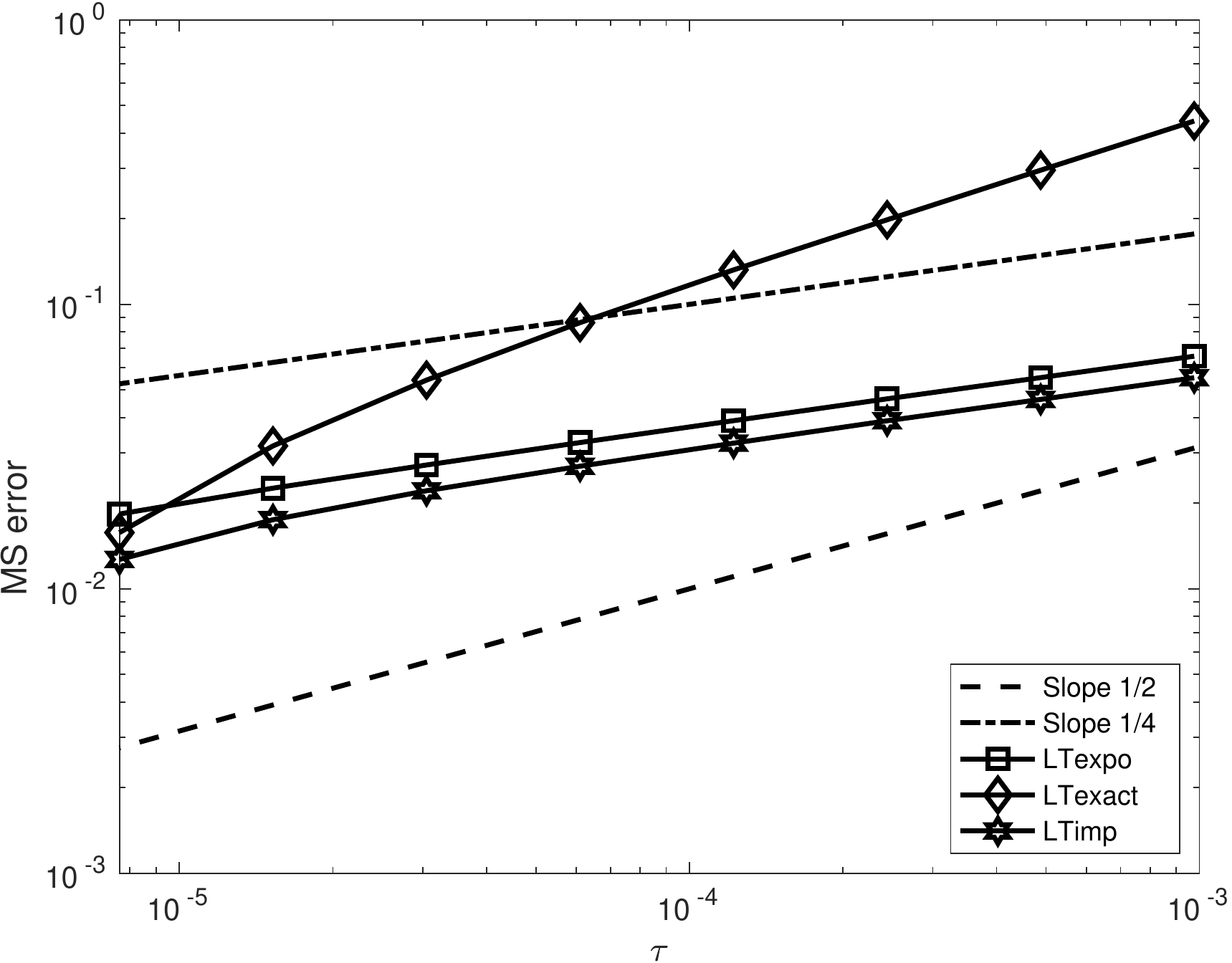}
\caption{Mean-square errors as a function of the time step:
Lie--Trotter splitting schemes: left ($\phi_\tau=\phi_\tau^{\rm L}\circ\phi_\tau^{\rm NL}$)
and right ($\phi_\tau=\phi_\tau^{\rm NL}\circ\phi_\tau^{\rm L}$) ($\diamond$ for \textsc{LTexact},
$\square$ for \textsc{LTexpo}, $\smallstar$ for \textsc{LTimp}).
The dotted lines have slopes $1/2$ and $1/4$.}
\label{fig:params}
\end{figure}

\section*{Acknowledgements}
The work of CEB is partially supported by the following project SIMALIN (ANR-19-CE40-0016) operated by the French National Research Agency.
The work of DC is partially supported by the Swedish Research Council (VR) (projects nr. $2018-04443$).
The computations were performed on resources provided by the Swedish National Infrastructure
for Computing (SNIC) at UPPMAX, Uppsala University.

\begin{appendix}
\section{Proof of the inequality~\eqref{eq:ineq}}\label{appendixx}

Let us first state two elementary inequalities:
\begin{itemize}
\item for all $0\le a\le b$ and $n\in\N$, one has
$
0\le b^n-a^n\le nb^{n-1}(b-a),
$
\item for all $z\in[0,\infty)$, one has
$
0\le \frac{1}{1+z}-e^{-z}\le C\min(1,z^2).
$
\end{itemize}
As a consequence, for all $n\in\N$ and $z\in[0,\infty)$ one has
\[
0\le \frac{1}{(1+z)^n}-e^{-nz}\le \frac{n}{(1+z)^{n-1}}\bigl(\frac{1}{1+z}-e^{-z}\bigr).
\]
\begin{proof}[Proof of~\eqref{eq:ineq}]
For all $n\ge 3$ and $z\in[0,\infty)$, one has
\begin{align*}
n|\frac{1}{(1+z)^n}-e^{-nz}|&\le \frac{Cn^2z^2}{(1+z)^{n-1}}\le \frac{Cn^2z^2}{1+(n-1)z+\frac{(n-1)(n-2)}{2}z^2}\le \frac{2Cn^2}{(n-1)(n-2)}\le C.
\end{align*}
The cases $n=1$ and $n=2$ are treated separately, one has
\[
\underset{z\in[0,\infty)}\sup~|\frac{1}{(1+z)}-e^{-z}|+\underset{z\in[0,\infty)}\sup~2|\frac{1}{(1+z)^2}-e^{-2z}|<\infty.
\]
This concludes the proof of the first inequality.
To prove the second inequality, observe first that one has
\[
\underset{n\in\N,z\in[0,\infty)}\sup~|\frac{1}{(1+z)^n}-e^{-nz}|\le 2.
\]
In addition, for all $n\ge 2$ and $@miscz\in[0,\infty)$, one has
\begin{align*}
\frac{|\frac{1}{(1+z)^n}-e^{-nz}|}{z}&\le \frac{Cnz}{(1+z)^{n-1}}\le \frac{Cnz}{1+(n-1)z}\le \frac{Cn}{n-1}\le C.
\end{align*}
The case $n=1$ is treated separately: using the inequality $\min(1,z^2)\le z$ one has
\[
\underset{z\in[0,\infty)}\sup~\frac{|\frac{1}{1+z}-e^{-z}|}{z}\le C.
\]
Gathering the results concludes the proof of the second inequality.
\end{proof}
\end{appendix}


\end{document}